\documentclass{amsart}
\usepackage{amssymb,amsthm,mathdots,tikz,amsmath}

\newtheorem{thm}{Theorem}[section]
\newtheorem{cor}[thm]{Corollary}
\newtheorem{lem}[thm]{Lemma}
\newtheorem{prop}[thm]{Proposition}
\newtheorem{conjecture}[thm]{Conjecture}

\theoremstyle{definition}

\theoremstyle{remark}
\newtheorem{rem}[thm]{Remark}
\newtheorem{remark}[thm]{Remark}
\numberwithin{equation}{section}
\newcommand{\To}{\longrightarrow}

\newcommand{\inv}{^{-1}}
\newcommand{\C}{\mathbb C}
\newcommand{\Z}{\mathbb Z}
\newcommand{\R}{\mathbb R}
\newcommand{\N}{\mathbb N}

\newcommand{\Y}{\mathcal Y}

\newcommand{\af}{\mathrm{af}}

\newcommand{\geh}{\mathfrak g}
\newcommand{\KK}{\mathcal K}
\newcommand{\MV}{\mathrm{MV}}
\newcommand{\OO}{\mathcal O}
\newcommand{\Perv}{\mathrm{Perv}}
\newcommand{\q}{\mathrm{qSchubert}}

\newcommand{\Rep}{\mathrm{Rep}}

\newcommand{\ip}[1]{\langle #1 \rangle}

\newcommand{\x}{\times}
\newcommand{\Gr}{\operatorname{Gr}}
\newcommand{\Spec}{\operatorname{Spec}}

\newcommand{\la}{\lambda}
\newcommand{\al}{\alpha}
\newcommand{\ep}{\varepsilon}
\newcommand{\si}{\sigma}

\newcommand{\Ad}{\operatorname{Ad}}
\newcommand{\sign}{\operatorname{sign}}
\newcommand{\Sym}{\operatorname{Sym}}

\newcommand{\height}{\operatorname{ht}}

\begin{document}

\title{Total positivity, Schubert positivity, and Geometric Satake}
\author{Thomas Lam}%
\address{Department of Mathematics, University of Michigan, Ann Arbor, MI
  48109 USA}%
\email{tfylam@umich.edu}%

\author{Konstanze Rietsch}%
\address{King's College London}%
\email{konstanze.rietsch@kcl.ac.uk}%


 \thanks{TL was supported by NSF grants DMS-0901111 and DMS-1160726, and by a Sloan Fellowship.}
\thanks{KR was funded by
EPSRC grant EP/D071305/1.}


\begin{abstract}
Let $G$ be a simple and simply-connected complex algebraic group, and let $X \subset G^\vee$ be the centralizer subgroup of a principal nilpotent element.  Ginzburg and Peterson independently related the ring of functions on $X$ with the homology ring of the affine Grassmannian $\Gr_G$.  Peterson furthermore connected this ring to the quantum cohomology rings of partial flag varieties $G/P$.

The first aim of this paper is to study three different notions of positivity on $X$: (1) {\it Schubert positivity} arising via Peterson's work, (2) {\it total positivity} in the sense of Lusztig, and (3) {\it Mirkovic-Vilonen positivity} obtained from the MV-cycles in $\Gr_G$.  Our first main theorem establishes that these three notions of positivity coincide.  The second aim of this paper is to parametrize the totally nonnegative part of $X$, confirming a conjecture of the second author.  

In type A a substantial part of our results were previously established by the second author.  The crucial new component of this paper is the connection with the affine Grassmannian and the geometric Satake correspondence.
%
%
%
\end{abstract}
\maketitle
\section{Introduction}

Let $G$ be a simply connected, semisimple complex linear algebraic group, split over $\R$.
The Peterson variety $\mathcal Y$ may be viewed as the compactification of the stabilizer $X:=G^\vee_F$ of
a standard principal nilpotent $F$ in $(\mathfrak g^\vee)^*$ (with respect to the coadjoint
representation of $G^\vee$), which one obtains by embedding $X$ into the Langlands dual
flag variety $G^\vee/B_-^\vee$ and taking the closure there.

Ginzburg \cite{Gin:GS} and Peterson \cite{Pet:QCoh} independently showed that the coordinate ring $\mathcal O(X)$ of the variety $X$ was isomorphic to the homology $H_*(\Gr_G)$ of the affine Grassmannian $\Gr_G$ of $G$, and Peterson discovered
moreover that the compactification $\mathcal Y$ encodes the quantum
cohomology rings of all of the flag varieties $G/P$.  Peterson's remarkable work in particular
exhibited explicit homomorphisms between localizations of $qH^*(G/P,\C)$ and $H_*(\Gr_G,\C)$ taking quantum Schubert classes $\sigma_w^P$ to  affine homology Schubert classes $\xi_x$.  These homomorphisms were verified in  \cite{LaSh:QH}.

The first aim of this paper is to compare different notions of positivity for the real points of $X$: (i) the {\it affine Schubert positive} part $X^{\af}_{>0}$ where affine Schubert classes $\xi_x$ take positive values via Ginzburg and Peterson's isomorphism $H_*(\Gr_G) \simeq \mathcal O(X)$; (ii) the {\it totally positive} part $X_{>0} :=X \cap U^\vee_{-,>0}$ in the sense of Lusztig \cite{Lus:TotPos94}; and (iii) the {\it Mirkovic-Vilonen positive} part $X^{\MV}_{>0}$ where the classes of the Mirkovic-Vilonen cycles from the geometric Satake correspondence \cite{MV:GS} take positive values.  

Our first main theorem (Theorem \ref{thm:three}) states that these three notions of positivity coincide.  For $G$ of type $A$ the coincidence $X^{\af}_{>0} = X_{>0}$ was already established in \cite{Rie:QCohPFl}, where instead of $X^{\af}_{>0}$, the notion of {\it quantum Schubert positivity} was used.  In general quantum Schubert positivity is possibly weaker than affine Schubert positivity.  It follows from \cite{Rie:QCohPFl} that the notions coincide in type A, and we verify that they coincide in type C in Appendix \ref{s:proof1}.

Our second main theorem (Theorem \ref{t:main}) is a parametrization of the totally positive $X_{>0}$ and totally nonnegative $X_{\geq 0}$ parts of $X$.  We show that they are homeomorphic to $\R_{>0}^n$ and $\R_{\geq 0}^n$ respectively.  This was conjectured by the second author in \cite{Rie:QCohPFl} where it was established in type $A$.  In type $A_n$ we have that $X=G^\vee_F$ is the $n$-dimensional subgroup of lower-triangular unipotent Toeplitz matrices, and thus the parametrization $X_{\geq 0} \simeq \R_{\geq 0}^n$ is a ``finite-dimensional'' analogue of the Edrei-Thoma theorem \cite{Edr:ToeplMat} parametrizing {\it infinite} totally nonnegative Toeplitz matrices, appearing in the classification of the characters of the infinite symmetric group.  The results of this article give an arbitrary type generalization. 

The strategy of our proof is as follows: to show that $X^{\af}_{>0} \subseteq X^{\MV}_{>0}$ we use a result of Kumar and Nori \cite{KuNo:pos} stating that effective classes in $H_*(\Gr_G)$ are Schubert-positive.  We then use the geometric Satake correspondence \cite{Gin:GS,MV:GS,Lus:GS} to describe $X^{\MV}_{>0}$ via matrix coefficients, and a result of Berenstein-Zelevinsky \cite{BeZe:Chamber} to connect to the totally positive part $X_{>0}$.  

Finally, to connect $X_{>0}$ back to $X^{\af}_{>0}$, we parametrize the latter directly by combining the positivity of the 3-point Gromov-Witten invariants of $qH^*(G/B)$ with the Perron-Frobenius theorem.  This argument follows the strategy of \cite{Rie:QCohPFl}.

There is a general phenomenon \cite{Lus:TotPos94,BeZe:Chamber} that totally positive parts have ``nice parametrizations''.  This phenomenon is closely related to the relation between total positivity and the canonical bases \cite{Lus:Can}, and also the cluster algebra structures on related stratifications \cite{Fom}.  Indeed our work suggests that the coordinate ring $\mathcal O(X)$ has the affine homology Schubert basis $\{\xi_w\}$ as a ``dual canonical basis'', and that the Hopf-dual universal enveloping algebra $U(\geh^\vee_F)$ has the cohomology affine Schubert basis $\{\xi^w\}$ as a ``canonical basis''.  Certainly the affine Schubert bases have the positivity properties expected of canonical bases.

In \cite{Rie:TotPosGBCKS} the type A parameterization result for the totally positive part $X_{>0}$ of the 
Toeplitz matrices $X$ is proved in a completely different way, using a mirror symmetric construction of $X$.
This approach does not however prove the interesting positivity properties of the bases we study in this paper. 
The mirror symmetric approach was partly generalized to other types in \cite{Rie:QToda}, where the 
existence of a totally positive point in $X$ for any choice of positive quantum parameters is proved (but not its uniqueness).

\vskip.2cm {\it Acknowledgements.} The authors would
like to thank Dale Peterson for his beautiful results which underly
this work. The second author also thanks Dima Panov for some helpful
conversations. The authors thank Victor Ginzburg for a question which
led to the inclusion of Section~6.2.

\section{Preliminaries and notation}\label{s:Prelims}
Let $G$ be a simple linear algebraic group over
$\C$ split over $\R$. Usually $G$ will be simply connected. 
 Denote  by $\Ad:G\to GL(\mathfrak
g)$  the adjoint representation of $G$ on its Lie algebra $\mathfrak
g$. We fix opposite Borel subgroups $B^+$ and $B^-$ defined over
$\R$ and intersecting in a split torus $T$. Their Lie algebras
are denoted by $\mathfrak b^+$ and $\mathfrak b^-$ respectively. We
will also consider their unipotent radicals $U^+$ and $U^-$ with
their Lie algebras $\mathfrak u^+$ and $\mathfrak u^-$.

Let $X^*(T)$ be the character group
of $T$ and $X_*(T)$ the group of cocharacters together with the
usual perfect pairing $\ip{\, ,\,}: X^*(T)\x X_*(T)\to\Z$.  
We may identify $X^*(T)$ with a lattice inside $\mathfrak h^*$,
and $X_*(T)$ with the dual lattice inside $\mathfrak h$. These
span the real forms $\mathfrak h^*_\R$ and $\mathfrak h_\R$,
respectively. 

Let $\Delta_+\subset X^*(T)$ be the set of positive roots corresponding
to $\mathfrak b^+$, and $\Delta_-$ the set of negative roots. 
There is a unique highest root in $\Delta_+$
which is denoted by $\theta$.
Let $I=\{1,\dotsc, n\}$ be an indexing set for the set $\Pi:=\{\alpha_i\ |\ i\in I\}$
of positive simple
roots . 
The $\alpha_i$-root space $\mathfrak g_{\alpha_i}\subset\mathfrak g$ is
spanned by Chevalley generator $e_i$ and $\mathfrak g_{-\alpha_i}$
is  spanned by $f_i$.  The split real form of $\mathfrak g$, denoted $\mathfrak g_\R$
is generated by the Chevalley generators $e_i,f_i$.

Let $Q:=\left < \alpha_1,\dotsc,\alpha_{n}\right >_\Z$ be the root lattice.
We also have the
fundamental weights $\omega_1,\dotsc, \omega_{n}$,  and the weight lattice
$L:=\left < \omega_1,\dotsc,\omega_{n}\right >_\Z$  associated to $G$.  If $G$ is simply connected,  we have the relations
$$
Q\subset X^*(T)=L\subset \mathfrak h^*.
$$
Let $Q^\vee$ denote the lattice spanned by the simple coroots, 
$\alpha_1^\vee,\dotsc, \alpha_n^\vee$, and $L^\vee$ the lattice 
spanned by the fundamental coweights $\omega^\vee_1,\dotsc,\omega^\vee_{n}$.
Then $Q^\vee$ is the dual lattice to $L$ and $L^\vee$ the dual lattice to $Q$,
giving   
$$
Q^\vee=X_*(T)\subset L^\vee\subset \mathfrak h,
$$
in the case where $G$ is simply connected. We set $\rho=\sum_{i\in I}\omega_i$ and
write $\operatorname{ht}(\lambda^\vee)=\left<\rho,\lambda^\vee\right>$ for the 
height of $\lambda^\vee\in Q^\vee$.


For any Chevalley generator $e_i, f_i$ of $\mathfrak g$ we may define
 a `simple root subgroup' by
  $$
  x_i(t)=\exp(t e_i),\qquad y_i(t)=\exp(t f_i),\qquad\qquad \text{for $t\in\C$.}
 $$

Let $W=N_G(T)/T$ be the Weyl group of $G$. It is generated by simple
reflections $s_1,\dotsc, s_{n}$. The length function $\ell: W\to \N$
gives the length of a reduced expression of $w\in W$ in the simple
reflections. The unique longest element is denoted $w_0$, and for a 
root $\alpha$, we let $r_\alpha$ denote the corresponding reflection.  For any
simple reflection $s_i$ we choose a representative $\dot s_i$ in $G$
defined by
$$
\dot s_i:=x_i(-1) y_i(1)x_i(-1).
$$
If $w=s_{i_1}\dotsc s_{i_m}$ is a reduced expression, then $\dot
w:=\dot s_{i_1}\dotsc \dot s_{i_m}$ is  a well-defined
representative for $w$, independent of the reduced expression
chosen. $W$ is a poset under the Bruhat order $\le$.

We denote the Langlands dual group of $G$ by $G^\vee$, or $G^\vee_\C$
to emphasize that we mean the algebraic group over $\C$.  
The notations for $G^\vee$ are the same as those for $G$ but with added
${}^\vee$ and any other superscripts moved down, for example $B^\vee_+$ for the
analogue of $B^+$.

\subsection{Parabolic subgroups}\label{s:parabolics}
Let $P$ denote a parabolic subgroup of $G$ containing $B^+$, and
let $\mathfrak p$ be the Lie algebra of $P$. Let $I_P$ be the subset
of $I$ associated to $P$ consisting of all the $i\in I$ with
$\dot s_i\in P$ and consider its complement $I^P:=I\setminus I_P$. 

Associated to $P$ we have the parabolic subgroup
$W_P=\left<s_i\ |\ i\in I_P\right>$ of $W$.
We let $W^P\subset W$ denote the set of minimal coset representatives
for $W/W_P$. An element $w$ lies
in $W^P$ precisely if for all reduced expressions
$w=s_{i_1}\cdots s_{i_m}$ the last index $i_m$ always lies in
$I^P$. We write $w^P$ or $w_0^P$ for the longest element in $W^P$,
while the longest element in $W_P$ is denoted $w_P$. For example
$w^B_0=w_0$ and $w_B=1$. Finally $P$ gives rise to a decomposition
\begin{equation*}
\Delta_+= \Delta_{P,+}\sqcup \Delta_{+}^P.
\end{equation*}
Here $\Delta_{P,+}=\{\al\in\Delta_+ \ | \ \langle\al,
\omega_i^\vee\rangle=0
\text{ all $i\in I^P$} \} $, so that
$$
\mathfrak p=\mathfrak b^+\oplus\bigoplus_{\alpha\in\Delta_{P,+}}\mathfrak g_{-\alpha},
$$
and $\Delta_+^P$ is the complement of $\Delta_{P,+}$ in $\Delta_+$.  
For
example $\Delta_{B,+}=\emptyset$ and $\Delta^B_+=\Delta_+$.

\section{Total Positivity}\label{s:totpos}
\subsection{Total positivity}
A matrix $A$ in $GL_n(\R)$ is called {\it totally positive} (or {\it
totally nonnegative}) if all the minors of $A$ are positive
(respectively nonnegative). In other words $A$ acts by positive or
nonnegative matrices in all of the fundamental representations
$\bigwedge^k\R^n$ (with respect to their standard bases). In the
1990's Lusztig  \cite{Lus:TotPos94} extended this theory dating back
to the 1930's to all reductive algebraic groups. This work followed
his construction of canonical bases and utilized their deep
positivity properties in types ADE.

Let $G$ be a simple algebraic group, split over the
reals.  For the rest of this paper the definitions here will be applied to $G^\vee$ rather than $G$.

The totally nonnegative part $U^+_{\geq 0}$ of $U_+$ is the
semigroup generated by $\{x_i(t) \mid i \in I \; \text{and} \; t \in
\R_{\geq 0}\}$.  Similarly the totally nonnegative part $U^-_{\geq
0}$ of $U_-$ is the semigroup generated by $\{y_i(t) \mid i \in I \;
\text{and} \; t \in \R_{\geq 0}\}$.  The totally positive parts are
given by $U^+_{>0} = U^+_{\geq 0} \cap B^-\dot w_0 B^-$ and
$U^-_{>0} = U^-_{\geq 0} \cap B^+\dot w_0 B^+$.
%

\subsection{Matrix coefficients}
Suppose $\lambda \in X^*(T)$ is dominant.  Then we have a highest weight
irreducible representation $V_\lambda$ for $G$. The Lie algebra 
$\mathfrak g$ also acts on 
$V_\lambda$ as does its universal enveloping algebra $U(\mathfrak g)$.
We fix a highest weight vector $v_\lambda^+$ in $V_\lambda$. 
The vector space $V_\lambda$ has a real form given by
$V_{\lambda,\R}=U(\mathfrak g_{\mathbb R})\cdot v^+_\lambda$.

 Let $(\ )^T: U(\mathfrak g)\to  U(\mathfrak g)$ be the
unique involutive anti-automorphism satisfying $e_i^T=f_i$.
We let $\ip{.,.}:  V_\lambda\times V_\lambda \to \C$ denote the unique symmetric, non-degenerate 
bilinear form (Shapovalov form)\cite[II, 2.3]{Kum:Book} satisfying 
\begin{eqnarray} \label{E:adjoint}
 \ip{u \cdot v,v'} &= & \ip{v,u^T \cdot v'} \qquad \text{for all $u\in U(\mathfrak g), v,v'\in V_\lambda$,}
\end{eqnarray}
normalized so that $\ip{v^+_\lambda,v^+_\lambda} =1$. The Shapovalov form is
real positive definite on $V_{\lambda,\R}$, see \cite[Theorem 2.3.13]{Kum:Book}.

 We will be
studying total positivity in the Langlands dual group $G^\vee$ of a
simply-connected group $G$.  Thus $G^\vee$ will be adjoint. Let
$G^*$ be the simply-connected cover of $G^\vee$. Then the unipotent
subgroups of $G^*$ and $G^\vee$ can be identified, and so can their
totally positive (resp.~negative) parts.  The purpose of this
observation is to allow the evaluation of matrix coefficients of
fundamental representations on the unipotent subgroup of $G^\vee$.
(The adjoint group $G^\vee$ itself may not act on these
representations.)

Thus for a fundamental weight $\omega_i$ (not necessarily a
character of $G$!) and a vector $v \in V_{\omega_i}$ we have a
matrix coefficient
$$
y\longmapsto \ip{v, y \cdot v^+_{\omega_i}}
$$
on $U^-$. 
The following result follows from a theorem (\cite[Theorem 1.5]{BeZe:Chamber}) of Berenstein and Zelevinsky (note that every chamber weight is a $w_0$-chamber weight in the terminology of \cite{BeZe:Chamber}).
\begin{prop}\label{P:BZ}
Let $y \in U^-$.  Then $y$ is totally positive if and only if for
any $i \in I$ we have
$$
\ip{\dot w\cdot v^+_{\omega_i}, y\cdot v^+_{\omega_i} }> 0
$$
for each $w \in W$, where $v^+_{\omega_i}$ denotes a highest
weight vector in the irreducible highest weight representation
$V_{\omega_i}$.  
\end{prop}

We will need the following generalization of the above Proposition. 

\begin{prop}\label{P:genBZ}
Let $y \in U^-$. Suppose for any irreducible 
representation  $V_\lambda$ of $G$
with highest weight vector $v^+_\lambda$, 
and any weight vector $v$ which lies in a 
one-dimensional weight space of $V_{\lambda}$ 
such that $\ip{v,x \cdot v^+_{\lambda}}>0$ for all totally positive $x \in U^{-}_{>0}$ we have
\begin{equation}\label{e:allowable}
\ip{v, y\cdot v^+_{\lambda} }> 0.
\end{equation}
Then $y$ is totally positive.
\end{prop}

The proof of Proposition~\ref{P:genBZ} is delayed until Section \ref{s:genBZ}.  If $G$ is simply connected then this Proposition \ref{P:genBZ} follows from 
Proposition~\ref{P:BZ}.   The difference arises if $G$ is not simply connected, in which case the fundamental weights may not be
characters of the maximal torus of $G$. 

\begin{remark}\label{rem:can}
Suppose $G$ is simply-laced.  Then the matrix coefficients of $x \in U^{-}_{>0}$ in the canonical basis of any irreducible representation $V_\lambda$ are positive.  It follows that for any $v \neq 0$ lying in a one-dimensional weight space of $V_\lambda$, either $v$ or $-v$ has the property that $\ip{v,x \cdot v^+_{\lambda}}>0$ for all $x \in U^-_{>0}$.
\end{remark}

\section{The affine Grassmannian and geometric Satake}
In this section, $G$ is a simple simply-connected linear algebraic group over
$\C$.  Let $\OO = \C[[t]]$ denote the ring of formal power series and $\KK
=\C((t))$ the field of formal Laurent series.  Let $\Gr_G =
G(\KK)/G(\OO)$ denote the affine Grassmannian of $G$.  

\subsection{Affine Weyl group}\label{s:affWeylGroup}
Let $W_\af = W \ltimes X_*(T)$ be the affine Weyl group of $G$.  For
a cocharacter $\lambda \in X_*(T)$ we write $t_\lambda \in W_\af$
for the translation element of the affine Weyl group.  We then have
the commutation formula $wt_\lambda w^{-1} = t_{w \cdot \lambda}$.
The affine Weyl group is also a Coxeter group, generated by simple
reflections $s_0,s_1,\ldots,s_n$, where $s_0=r_\theta t_{-\theta^\vee}$.
It is a graded poset with its usual length function $\ell:W_{\af}\to \Z_{\ge 0}$, 
and Bruhat order $\ge$.

Let $W_\af^-$ denote the minimal length coset representatives of
$W_\af/W$.  Thus we have canonical bijections
\begin{equation}\label{e:Wafminus}
X_*(T) \longleftrightarrow W_\af/W \longleftrightarrow W_\af^-.
\end{equation}
The intersection $X_*(T) \cap W_\af^-$ is given by the anti-dominant
translations, that is $t_\lambda$ where $\ip{\alpha_i,\lambda} \leq
0$ for each $i \in I$.

Note that an element $\lambda$ of $X_*(T)$ viewed as a map from $\C^*$ to $T$
can also be reinterpreted as an element of  $T(\KK)$. We denote this element
by $t^\lambda$. The two should not be confused since the isomorphism
$W_{\af}\to N_{G(\KK)}(T)/T$ sends $t_{\lambda}$ to $t^{-\lambda}$.

\subsection{Geometric Satake and Mirkovic-Vilonen cycles}
The affine Grassmannian is an ind-scheme \cite{Kum:Book,Gin:GS,MV:GS}.  The $G(\OO)$-orbits
$\Gr_\lambda$ on $\Gr_G$ are parametrized by the dominant cocharacters
$\lambda \in X^+_*(T)$. Namely,
$$
\Gr_\lambda:=G(\OO)t^{\lambda}G(\OO)/G(\OO).
$$

The geometric Satake correspondence \cite{Gin:GS, Lus:GS, MV:GS}
(with real coefficients) states that the tensor category
$\Perv(\Gr_G)$ of $G(\OO)$-equivariant perverse sheaves on $\Gr_G$
with $\C$-coefficients is equivalent to the tensor category
$\Rep(G^\vee_\C)$ of finite-dimensional representations of the Langlands dual group $G^\vee_\C$.  (For our
purposes the tensor structure will be unimportant.)  The simple
objects of $\Perv(\Gr_G)$ are the intersection cohomology complexes
$IC_\lambda$ of the $G(\OO)$-orbit closures $\overline{\Gr}_\lambda$.
They correspond under the geometric Satake correspondence to the
highest weight representations $V_\lambda$ of $G^\vee$.
Furthermore, we have a canonical isomorphism
\begin{equation}\label{E:IC}
IH^*(\overline{\Gr}_\lambda) = H^*(\Gr_G, IC_\lambda) \simeq
V_\lambda.
\end{equation}
Mirkovic and Vilonen found explicit cycles in $\Gr_G$ whose
intersection homology classes give rise to a weight-basis of
$V_\lambda$ under the isomorphism $\eqref{E:IC}$.    We denote by
$\MV_{\lambda,v}$ the MV-cycle with corresponding  vector $v \in V_\lambda$.
For  $w \in W$, the weight-space
$V_\lambda(w\lambda)$ is one-dimensional.  We denote by
$\MV_{\lambda,w{\lambda}}$ the corresponding MV-cycle.  Thus
$[\MV_{\lambda,w\lambda}]_{IH} \in IH^*(\overline{\Gr}_\lambda) \simeq
V_\lambda$ has weight $w\lambda$.  All the statements of this section hold with $\R$-coefficients: we take perverse sheaves with $\R$-coefficients, and consider the representations of a split real form $G^\vee_\R$ of the Langlands dual group.

\subsection{Schubert varieties in $\Gr_G$}
Let $\mathcal I \subset G(\OO)$ denote the Iwahori
subgroup of elements $g(t)$ which evaluate to $g\in B^+$ at $t=0$. 
The $\mathcal I$-orbits $\Omega_\mu$ on $\Gr_G$, called {\it
Schubert cells}, are labeled by all (not necessarily dominant)
cocharacters $\mu \in X_*(T)$. Explicitly,
$$\Omega_\mu=\mathcal I\, t^{\mu} G(\mathcal O)/G(\mathcal O).
$$  
Alternatively, we may label
Schubert cells by cosets $xW \in W_\af/W$ or minimal coset
representatives $x \in W_\af^-$, using the bijection \eqref{e:Wafminus}.
Choosing a representative $\dot x$ of $x$ we have 
$$
\Omega_x=\mathcal I\, \dot x G(\mathcal O)/G(\mathcal O).
$$
The Schubert cell $\Omega_\mu= \Omega_x$ is isomorphic 
to $\C^{\ell(x)}$ whenever $x \in
W_\af^-$. We note that $\Omega_\mu=\Omega_{t_{-\mu}}$ if $\mu$ is dominant, compare
Section~\ref{s:affWeylGroup}. The Schubert varieties $X_x = \overline{\Omega_x}$, alternatively denoted $X_\mu=\overline{\Omega_\mu}$, are
themselves unions of Schubert cells: $X_x = \sqcup_{v \leq x}
\Omega_v$. The $G(\OO)$ orbits are also unions of Schubert
cells:
$$
\Gr_{\lambda}=\bigsqcup_{w\in W} \Omega_{w\cdot\lambda}.
$$
In particular the largest one of these, $\Omega_{\lambda}\cong \C^{\ell(t_{-\lambda})}$,
is open dense in $\Gr_\lambda$ (where we assumed $\lambda$ dominant), and so
\begin{equation}\label{e:GrLambdaIsSchub}
 \overline{\Gr}_{\lambda}=\overline{\Omega_{\lambda}}=X_{ \lambda}.
 \end{equation}
Thus every $G(\OO)$-orbit closure
is a Schubert variety, but not conversely.  
Moreover $\overline{\Gr}_{\lambda}$ has dimension $\ell(t_{-\lambda})$, which 
equals $2\operatorname{ht}(\lambda)$.
 
We note that the $MV$-cycle $MV_{\lambda,v}$ is an irreducible
subvariety of $\overline{\Gr_\lambda}$ of dimension $\height(\lambda)+\height(\nu)$
if $v$ lies in the $\nu$-weight space of $V_\lambda$, see \cite[Theorem~3.2]{MV:GS}. In particular 
$MV_{\lambda,\lambda}=\overline{\Gr_\lambda}$ and  $MV_{\lambda,w_0\lambda}$
is just a point.


\subsection{The (co)homology of $\Gr_G$}
The space $\Gr_G$ is homotopic to the based loop group
$\Omega K$ of polynomial maps of $S^1$ into the compact form $K
\subset G$ \cite{Qui:unp,PS:loopgroups}.  Thus the homology
$H_*(\Gr_G; \C)$ and cohomology $H^*(\Gr_G; \C)$ are commutative and
co-commutative graded dual Hopf algebras over $\C$.  

Ginzburg
\cite{Gin:GS} (see also \cite{BFM:K-hom}) and Dale Peterson \cite{Pet:QCoh} described $H_*(\Gr_G, \C)$
 as the coordinate ring of the stabilizer subgroup 
of a principal nilpotent in $\mathfrak (g^\vee)^*$. Namely, in our conventions, let
$F\in (\mathfrak g^\vee)^*$ be the principal nilpotent
element defined by
$$F=\sum_{i\in I} (e_i^\vee)^*,
$$
where  $(e_i^\vee)^*(\zeta )=0$
if $\zeta \in\mathfrak g_\alpha^\vee$ for $\alpha\ne\alpha_i$,
and $(e_i^\vee)^*(e_i^\vee)=1$. 
Let $X= (G^\vee)_F$ denote the stabilizer of $F$ inside
$G^\vee$, under the coadjoint action.  It is an abelian
subgroup of $U^\vee_-$ of dimension equal to the rank of $G$.
Then the result from  \cite{Gin:GS,Pet:QCoh} says that  $H_*(\Gr_G)$ is Hopf-isomorphic to the ring of
regular functions on $X$. 
Moreover, the cohomology, $H^*(\Gr_G,\C)$ is  Hopf-isomorphic to the 
universal enveloping algebra $U(\geh^\vee_F)$ of the centralizer of $F$,
as graded dual.

We note that Ginzburg \cite{Gin:GS} works over $\C$  while Peterson
 \cite{Pet:QCoh} works over $\Z$, but the details of Peterson's work
 are so far unpublished. 

Our choice  of principal nilpotent $F$ is
compatible via Peterson's isomorphism \eqref{e:PetIsoLoopHom}, see \cite{LaSh:QH},
with the conventions in 
\cite{Kos:QCoh,Kos:QCoh2,Rie:MSgen}, and is
related to the choice in \cite{Gin:GS,Pet:QCoh} by switching the roles of $B^+$
and $B^-$.

In terms of the above presentation
of $H_*(\Gr_G)$, the fundamental class of an MV-cycle can be described
as follows.  Let $\ip{.,.}:H^*(\Gr_G) \times H_*(\Gr_G) \to \C$ be the
pairing obtained from cap product composed with pushing forward to a
point.

\begin{prop}\label{p:MVmatrix} 
Suppose $\MV_{\lambda,v}$ is the MV-cycle with corresponding 
weight vector $v \in V_\lambda$ under \eqref{E:IC}.  Let $u \in
U(\geh^\vee_F) \simeq H^*(\Gr_G)$.  Then the fundamental class
$[\MV_{\lambda,v}] \in H_*(\Gr_G)$ satisfies
$$
\ip{u,[\MV_{\lambda,v}]} = \ip{ u\cdot v,v^-_\lambda},
$$
where $v^-_\lambda$ is the lowest weight vector of $V_\lambda$ (in the MV-basis).
\end{prop}
\begin{proof}
The argument is essentially the same as \cite[Proposition
1.9]{Gin:GS}; the main difference is that in our conventions $u$ is lower unipotent, 
rather than upper unipotent, however accordingly $\overline{\Gr_\lambda}$ is 
in our conventions the MV-cycle representing the 
highest weight vector, whereas it is the lowest weight 
vector in \cite{Gin:GS}.  So the difference is that everywhere
the roles of $B^+$ and $B^-$ are interchanged. 
 By \cite[Theorem 1.7.6]{Gin:GS}, the action of
$u\in U(\geh^\vee_F)$ on $V_\lambda$ is compatible with the action
of the corresponding element in $H^*(\Gr_G)$ on $IH^*(\overline{\Gr}_\lambda)$. 
Under \eqref{E:IC}, the vector $v$ is sent to $[\MV_{\lambda,v}]_{IH}$
which maps to the fundamental class $[\MV_{\lambda,v}]$ under the natural 
map from the intersection cohomology $IH^*(\overline{\Gr}_\lambda)$
to the homology $H_*(\overline{\Gr}_\lambda)$.
Also, under the fundamental class map the
action of $H^*(\Gr_G)$
on $IH^*(\overline{\Gr}_\lambda)$ is sent to
the cap product of $H^*(\Gr_G)$ on $H_*(\overline{\Gr}_\lambda)$.
Finally, pushing forward to a point is the same as pairing with
$v^-_\lambda$ (in our conventions). So we get the identity
$$
\ip{u,[\MV_{\lambda,v}]}= \pi_*(u\cap [\MV_{\lambda,v}]) = \ip{v^-_\lambda,  u\cdot v}.
$$
where $\pi:X\to\{\, pt\}$.
\end{proof}

\subsection{Schubert basis}
We have
$$
H_*(\Gr_G) = \bigoplus_{x \in W_\af^-} \C \cdot \xi_x, \ \ \ \ H^*(\Gr_G)
= \bigoplus_{x \in W_\af^-} \C \cdot \xi^x,
$$
where the $\xi_w$ are the fundamental classes $[X_w]$ of the
Schubert varieties, and $\{\xi^w\}$ is the cohomology basis (dual
under the cap product).  Suppose $\lambda$ is dominant, then we also have
\begin{equation*}
H_*(\overline{\Gr}_\lambda)=\bigoplus_{\small\begin{matrix} x \in W_\af^- \\ x\le t_{-\lambda} \end{matrix}} \C \cdot \xi_x,
\end{equation*}
because of \eqref{e:GrLambdaIsSchub} and the decomposition of $X_{\lambda}$ into Schubert cells.

By Ginzburg/Peterson's isomorphism, we will often
think of a Schubert basis element $\xi_w$ as a function on $X$.  
The Schubert basis of $H_*(\Gr_G)$ has the following factorization property:

\begin{prop}[{\cite{Pet:QCoh,LaSh:QH}}]
Suppose $wt_\nu, t_\mu \in W_\af^-$.  Then
$\xi_{wt_\nu}\xi_{t_\mu} = \xi_{wt_{\nu+\mu}}$.
\end{prop}
We remark that if $wt_\nu \in W_\af^-$, then necessarily $\nu$ is anti-dominant.

\section{The quantum cohomology ring of $G/P$}

\subsection{The usual cohomology of $\mathbf{G/P}$ and its Schubert
basis~}\label{s:ClassCoh} For our purposes it will suffice to
take homology or cohomology with complex coefficients, so
$H^*(G/P)$ will stand for $H^*(G/P,\C)$. By the well-known result of C.~Ehresmann,
the singular homology of $G/P$ has a
basis indexed by the elements $w\in W^P$ made up of the
fundamental classes of the Schubert varieties,
\begin{equation*}
X^P_w:=\overline{(B^+wP/P)}\subseteq G/P.
\end{equation*}
Here the bar stands for (Zariski) closure. Let $\si^P_{w}\in
H^*(G/P)$ be the Poincar\'e dual class to $[X^P_w]$. Note
that $X^P_w$ has complex codimension $\ell(w)$ in $G/P$ and
hence $\sigma^P_w$ lies in $H^{2\ell(w)}(G/P)$. The set
$\{\si^P_{w}\ |\ w\in W^P\}$ forms a basis of $H^*(G/P)$ called
the Schubert basis. The top degree cohomology of $G/P$ is spanned
by $\si^P_{w_0^P}$ and we have
the Poincar\'e duality pairing
\begin{equation*}
 H^*(G/P)\x H^*(G/P)\To \C ,\qquad(\si,\mu)\mapsto \left<\si\cup \mu\right>
\end{equation*}
which may be interpreted as taking $(\si,\mu)$ to the coefficient
of $\si^P_{w^P_0}$ in the basis expansion of the product $\si\cup
\mu$. For $w\in W^P$ let $PD(w)\in W^P$ be the minimal length
coset representative in $w_0wW_P$. Then this pairing is
characterized by
\begin{equation*}
\left <\si^P_w\cup \si^P_v\right>=\delta_{w, PD(v)}.
\end{equation*}

\subsection{The quantum cohomology ring $\mathbf{qH^*(G/P)}$}
The (small) quantum cohomology ring $qH^*(G/P)$ is a deformation of
the usual cohomology ring by $\C[q^P_1,\dotsc, q^P_k]$, where $k=\dim H^2(G/P)$,
with structure constants defined by $3$-point genus $0$ Gromov-Witten
invariants.  For more background on quantum cohomology, see \cite{FuWo:SchubProds}.

We have
$$
qH^*(G/P) = \oplus_{w \in W^P} \C[q^P_1,\dotsc, q^P_k] \cdot \si^P_w
$$
where $\si^P_w$ now (and in the rest of the paper) denotes the quantum Schubert class.  The quantum
cup product is defined by
\begin{equation*}
\sigma^P_v\cdot \sigma^P_w=\sum_{\begin{smallmatrix}u\in W^P\\ \mathbf
d\in \mathbb N^k
\end{smallmatrix} }\left<\si_u^P,\si_v^P,\si_w^P\right>_{\mathbf
d}\ \mathbf q^\mathbf d\si^P_{PD(u)},
\end{equation*}
where $\mathbf
q^\mathbf d$ is multi-index notation for $\prod_{i=1}^k
q_i^{d_i}$, and the $\left<\si_u^P,\si_v^P,\si_w^P\right>_{\mathbf d}$
 are genus $0$, $3$-point Gromov-Witten invariants.  These enumerate rational curves in $G/P$, with a fixed degree determined by $\mathbf d$, which pass through generaic translates of three Schubert varieties.  In particular, $\left<\si_u^P,\si_v^P,\si_w^P\right>_{\mathbf d}$ is a  nonnegative integer.

The quantum cohomology ring $qH^*(G/P)$ has an analogue of the
Poincar\'e duality pairing which may be defined as the symmetric $\C[q^P_1,\dotsc,
q^P_k]$-bilinear pairing
\begin{equation*} qH^*(G/P)\x qH^*(G/P)\To
\C[q^P_1,\dotsc, q^P_k], \qquad (\si,\mu)\mapsto \left<\si\cdot
\mu\right>_{\mathbf q}
\end{equation*}
where $\left<\si\cdot \mu\right>_{\mathbf q}$ denotes the coefficient of
$\si^P_{w_0^P}$ in the Schubert basis expansion of the product
$\si \cdot \mu$.
In terms of the Schubert basis the quantum Poincar\'e duality
pairing on $qH^*(G/P)$ is given by
\begin{equation}\label{E:FW}
\left <\si^P_w\cdot\si^P_v\right>_{\mathbf q}=\delta_{w, PD(v)},
\end{equation}
where $v,w\in W^P$, and $PD:W^P\to W^P$ is the involution defined
in Section \ref{s:ClassCoh}.  Equation \eqref{E:FW} can for example be deduced from Fulton and Woodward's results on the minimal coefficient of $q$ in a quantum product.\footnote{We thank L.~Mihalcea  for pointing out that it also follows from Proposition 3.2 of ``Finiteness of cominuscule quantum $K$-theory'' by Buch, Chaput, Mihalcea, and Perrin.}

%
%
%

\section {Peterson's theory}\label{PetersonTheory}
In this section we summarize Peterson's results concerning his
geometric realizations of $qH^*(G/P)$ and their relationship with
$H_*(\Gr_G)$.

\subsection{Definition of the Peterson variety.}\label{s:Peterson}

Each $\Spec(qH^*(G/P))$ turns out to be most naturally viewed as
lying inside  the Langlands dual flag variety $G^\vee/B^\vee$, where
it appears as a stratum (non-reduced intersection with a Bruhat
cell) of one $n$-dimensional projective variety called the {\it
Peterson variety}. This remarkable fact was discovered and shown by
Dale Peterson \cite{Pet:QCoh}.

The condition
\begin{equation*}
(\Ad(g\inv)\cdot F)(X)=0\ \text{ for all $X\in[\mathfrak u^\vee_-,\mathfrak u^\vee_-]$,}
\end{equation*}
defines a closed subvariety of $G^\vee$ invariant under right
multiplication by $B_-^\vee$. Thus they define a closed subvariety of
$G^\vee/B_-^\vee$. This subvariety $\Y$ is the {\it Peterson variety}
for $G$. Explicitly we have
\begin{equation*}
\Y=\left\{gB_-^\vee\in G^\vee/B_-^\vee\ \left |\ \Ad(g\inv)\cdot F \in [\mathfrak u^\vee_-,\mathfrak u^\vee_-]^\perp \right.\right\}.
\end{equation*}
For any parabolic subgroup $W_P\subset W$ with longest element
$w_P$ define $\Y_P$ as non-reduced intersection,
\begin{equation*}
\Y_P:= \Y\x_{G^\vee/B_-^\vee} \left(B_+^\vee w_P B_-^\vee/B_-^\vee\right ).
\end{equation*}

\begin{rem}
For $P=B$ we have a map
\begin{equation}
\label{e:IsoTodaLeaf}
\mathcal Y_B\to\mathcal A_G\ :\quad u B^\vee_- \mapsto u\inv \cdot F,
\end{equation}
where $\mathcal A_G\subset (\mathfrak g^\vee)^*$ is
the degenerate leaf of the Toda lattice.  
This 
map is an isomorphism
as follows from classical work of Kostant \cite{Kos:Toda}.  Kostant also showed that
$\mathcal Y_B$ is irreducible \cite{Kos:QCoh}.

 The isomorphism between $qH^*(G/B)$ and the functions on the degenerate leaf of the Toda lattice was 
established by B.~Kim \cite{Kim}  building on \cite{GiKi:FlTod}.
  \end{rem}

\subsection{Irreducibility of $\mathcal Y$.}\label{s:Peterson}

It is not immediately obvious from the above definition that the  Peterson variety $\mathcal Y$
is irreducible. In other words apart from the the closure of $\mathcal Y_B$
it could a priori contain some other irreducible components coming from intersections with other Bruhat cells. 
We include a sketch of proof (put together from \cite{Pet:QCoh}) that this doesn't happen, 
and that therefore $\mathcal Y$ is irreducible, $n$-dimensional and equal to the closure of $\mathcal Y_B$. 
Namely we have the following proposition.

\begin{prop}[Dale Peterson]  
If $w=w_P$, the longest element in $W_P$ for some parabolic subgroup $P$, then $\mathcal Y\cap B^\vee_+ w B^\vee_-/B^\vee_-$ is nonempty and  of
dimension $|I^P|$. Otherwise $w\inv\cdot (-\Pi^\vee)\not\subset \Delta^\vee_-\cup\Pi^\vee$
and  $\mathcal Y\cap B^\vee_+ w B^\vee_-/B^\vee_-=\emptyset$. 
\end{prop}


\bigskip
\begin{proof}[Sketch of proof]
Clearly $w\inv\cdot F$ needs to lie in $b_{\Pi}:= [\mathfrak u^\vee_-,\mathfrak u^\vee_-]^\perp$ for
$\mathcal Y\cap B^\vee_+ w B^\vee_-/B^\vee_-$ to be non-empty.  So $w\inv\cdot (-\Pi^\vee)\subset \Delta^\vee_-\cup\Pi^\vee$.
This is the case if and only if $w=w_P$ for some parabolic $P$, by a lemma from \cite{Pet:QCoh} reproduced in \cite[Lemma~2.2]{Rie:QCohGr}.

Consider the map
\begin{eqnarray*}
\psi:  U^\vee_+ &\to & (\mathfrak u^\vee_-)^*, \\
u &\mapsto & (u\inv\cdot F)|_{\mathfrak u^\vee_-}.
\end{eqnarray*}

The coordinate rings of $U_+^\vee $  and $ (\mathfrak u_-)^*$ are
polynomial rings. On $U_+^\vee$ consider the $\C^*$-action coming from conjugation 
by the one-parameter subgroup of $T^\vee$ corresponding to $\rho\in X_*(T^\vee)$. 
On $(\mathfrak u^\vee_-)^* $
let $\C^*$ act by $z\cdot X_{\alpha}=z^{<\alpha,\rho>+1}X_{\alpha}$
for $X_{\alpha}$ in the $\alpha$-weight space of $(\mathfrak u^\vee_-)^*$
and $\alpha\in\Delta^\vee_+$.
Then $\psi^*$ is a homomorphism of 
(positively) graded rings, namely it is straightforward to check that $\psi$ is $\C^*$-equivariant.

Also $\psi$ has the property that $\psi\inv(0)=\{0\}$ in terms of $\C$-valued points, or indeed over any algebraically closed field.
Peterson proves  this in \cite{Pet:QCoh} by considering the $B^\vee_-$-Bruhat decomposition intersected
with $U^\vee_+$. Namely, the only way  $\psi(u)$ can be $0$ for $u\in U^\vee_+\cap B^\vee_-\dot w B^\vee_-$
is if $w=e$, wherefore $u$ must be the identity element in $U^\vee_+$.

It follows from these two properties that $\psi$ is finite.  
For example by page 660 
in Griffiths-Harris and using the $\C^*$-action to go from the statement
locally around zero, to a global statement, or by another proposition in Peterson's lectures \cite{Pet:QCoh}. 

Let
$$
U_P:=\psi\inv((w_P\cdot  b_{\Pi})|_{\mathfrak u^\vee_-} ) =\{u\in U^\vee_+\ |\ (u\inv\cdot F)|_{\mathfrak u^\vee_-}\in   
(w_P\cdot b_{\Pi})|_{\mathfrak u^\vee_-}\}.
$$
Since  $u\inv\cdot F\in F+\mathfrak h + (\mathfrak u^\vee_-)^*$
for $u\in U^\vee_+$ and $F+\mathfrak h\subset w_P\cdot b_{\Pi}$, we can drop the restriction  to ${\mathfrak u^\vee_-}$ 
on both sides of the condition above, and we have a projection map
$$
U_P=\{u\in U^\vee_+\ |\ u\inv\cdot F\in   
w_P\cdot b_{\Pi}\}=\{u\in U^\vee_+\ |\ w_P\inv u\inv\cdot F\in b_{\Pi}\}\to \mathcal Y\cap B^\vee_+ w_P B^\vee_-/B^\vee_-
$$
taking $u\in U_P$ to $u w_P B^\vee_-$, which is a fiber bundle with fiber $\cong \C^{\ell(w_P)}$.

Since $\psi$ is finite the dimension of $U_P$ is equal to the dimension of the subspace 
$(w_P\cdot  b_{\Pi})|_{\mathfrak u^\vee_-}$ inside ${(\mathfrak u^\vee_-)^*} $. This dimension is just $|I^P|+\ell(w_P)$, by
looking at the weight space decomposition. 
So $\mathcal Y\cap B^\vee_+w_P B^\vee_-/B^\vee_-$ has dimension $| I^P|$.
\end{proof}

\subsection{Geometric realization of $qH^*(G/P)$}

Recall the stabilizer $X$ of the principal nilpotent $F$, which
is an $n$-dimensional abelian subgroup of $U_-^\vee$.
Using an idea of Kostant's 
\cite[page 304]{Kos:qToda}, the Peterson variety may also be understood as a compactification
of $X$. Namely,
$$
\Y=\overline{X\dot w_0 B^\vee_-/B^\vee_-}\subset G^\vee/B^\vee_-.
$$
For the parabolic $P$ let
$$
\Y_P^*:=\Y_P\x_{G^\vee/B_-^\vee} X\dot w_0 B^\vee_-/B^\vee_-=(X\x_{G^\vee} B^\vee_+\dot w_P\dot w_0 B^\vee_+)\dot w_0/B^\vee_- \cong X\x_{G^\vee}B^\vee_+\dot w_P\dot w_0 B_+^\vee,
$$
or equivalently,
$$
\Y_P^*=\Y_P\x_{G^\vee/B^\vee_-}B^\vee_-\dot w_0 B^\vee_-/B^\vee_-.  
$$
We define
$$
X_P:=X\x_{G^\vee}B^\vee_+\dot w_P\dot w_0 B_+^\vee,
$$
so that the above is an isomorphism $\Y_P^*\cong X_P$.

\begin{thm}[Dale Peterson]\label{t:Pet}\
\begin{enumerate}
\item The $\mathcal Y_P$ give rise to a decomposition
\begin{equation*}
\Y(\C)=\bigsqcup_{P} \Y_P(\C).
\end{equation*}
\item
For $P=B$  we have
\begin{equation}\label{e:PetIso1}
qH^*(G/B)\overset\sim\to \mathcal O(\Y_B),
\end{equation}
via the isomorphism \eqref{e:IsoTodaLeaf} of $\Y_B$ with the degenerate leaf
of the Toda lattice of $G^\vee$.
\item If $w\in W^P$, then the the function
$S^w\in \mathcal O(\Y_B)$ associated
to the Schubert class $\sigma_w$ defines a
regular function $S_P^w$ on $\mathcal O(\Y_P)$. There is
an (uniquely determined) isomorphism
$$
qH^*(G/P)\overset\sim \to \mathcal O(\Y_P)
$$
which takes $\sigma^P_w$ to $S_P^w$.
\item
The isomorphisms above restrict, to give
isomorphisms
$$
qH^*(G/P)[q_1\inv,\dotsc, q_k\inv\ ]\overset\sim\to \mathcal O(\Y_P^*).
$$
\end{enumerate}
\end{thm}

In particular, Theorem \ref{t:Pet}(1) gives
\begin{equation*}
X(\C)=\bigsqcup_{P} X_P(\C).
\end{equation*}

\subsection{Quantum cohomology and homology of the affine Grassmannian}
\begin{thm}[Dale Peterson] \label{t:PetIsoLoop}\
\begin{enumerate}
\item
The composition of isomorphisms
\begin{equation}\label{e:PetIsoLoopHom}
H_*(\Gr_G)[\xi_{(t_\lambda)}^{-1}] \cong \mathcal O(X_B)\cong \mathcal O(\Y_B^*)\cong qH^*(G/B)[q_1\inv,\dotsc, q_n\inv\ ]
\end{equation}
is given by
\begin{equation}\label{e:PetIsoSchub}
\xi_{wt_\lambda} \xi_{t_\mu}^{-1} \longmapsto
q_{\lambda-\mu}\sigma^B_w
\end{equation}
where $q_{\nu} = q_1^{a_1}q_2^{a_2} \cdots q_n^{a_n}$ if $\nu = a_1
\alpha_1^\vee + \cdots + a_n \alpha_n^\vee$.
\item
More generally, for an arbitrary parabolic $P$ the composition
\begin{equation}\label{e:PetIsoLoopHomP}
(H_*(\Gr_G)/J_P)[(\xi_{(\pi_P(t_\lambda))}^{-1}] \cong \mathcal O(X_P)\cong \mathcal O(\Y_P^*)\cong qH^*(G/P)[q_1\inv,\dotsc, q_k\inv\ ]
\end{equation}
is given by
\begin{equation}\label{e:PetIsoSchubP}
\xi_{w\pi_P(t_\lambda)} \xi_{\pi_P(t_\mu)}^{-1} \longmapsto
q_{\eta_P(\lambda-\mu)}\sigma^P_w
\end{equation}
where $J_P \subset H_*(\Gr_G)$ is an ideal, $\pi_P$ maps $W_\af$ to a subset $(W^P)_\af$, and $\eta_P$ is the natural projection $Q^\vee \mapsto Q^\vee/Q^\vee_P$ where $Q^\vee_P$ is the root lattice of $W_P$.
\end{enumerate}
\end{thm}

Lam and Shimozono \cite{LaSh:QH} verified that the maps
\eqref{e:PetIsoSchub} (resp. \eqref{e:PetIsoSchubP}) are isomorphisms from $H_*(\Gr_G)[\xi_{(t_\lambda)}^{-1}]$ to
$qH^*(G/B)[q_1\inv,\dotsc, q_n\inv\ ]$ (respectively from, in the parabolic case, $(H_*(\Gr_G)/J_P)[(\xi_{(\pi_P(t_\lambda))}^{-1})]$ to $qH^*(G/P)[q_1\inv,\dotsc, q_k\inv\ ]$).  
We do not review the definitions of $J_P$ and $\pi_P$ here, but refer the reader to \cite{LaSh:QH}.

\begin{rem}
In \cite{LaSh:QH} it is not shown that the isomorphism \eqref{e:PetIsoSchub} is the one induced by the geometry of $X$.  We sketch how this can be achieved.

First, Kostant \cite[Section 5]{Kos:QCoh} 
expresses the quantum parameters as certain ratios of `chamber minors' on $X$.  Ginzburg's \cite[Proposition 1.9]{Gin:GS} also expresses the translation affine Schubert classes $\xi_{t_{-\lambda}}$ as matrix coefficients, 
since $\xi_{t_{-\lambda}}=[X_\lambda]=[\overline{Gr_\lambda}]=[MV_{\lambda,\lambda}]$.  This allows one to compare $\xi_{t_{-\lambda}}$ with $q_{\lambda}$ as functions on $X_B$, and see that they agree. Namely both are equal to $x\mapsto \left<x\cdot v^+_{\lambda},v^-_{\lambda}\right>$.

Let $\lambda=m\omega_i^\vee$ be a positive multiple of $\omega_i^\vee$ contained in $Q^\vee$. We now compare the functions $\xi_{s_it_{-\lambda}}$ and $q_{-\lambda}\sigma_{s_i}$ on $X_B$.  For the function $\sigma_{s_i}$, Kostant gives a formula in \cite[(119)]{Kos:QCoh} as a ratio of matrix coefficients on $X$.  For the function $\xi_{s_it_{-\lambda}}$, one notes that since $\lambda$ is a multiple of $\omega_i^\vee$, then
$t_{-\lambda} \in W_\af^-$ covers only $s_i t_{-\lambda}$ in the
Bruhat order of $W_\af^-$.  It follows that
$H_{2\ell(t_{-\lambda})-2}(\overline{\Gr}_{\lambda})$ is one-dimensional,
spanned by $\xi_{s_it_{-\lambda}}$.  Similarly the weight space
$V_{\lambda}(\lambda-\alpha_i^\vee)$ is one-dimensional. If we let $[\MV_{ \lambda,v}]$ be the unique MV-cycle
with $v$ of weight $\nu= \lambda-\alpha_i^\vee$, then this gives a cycle in 
homology of degree
$2(\operatorname{ht}(\lambda) + \operatorname{ht}(\nu))=4\operatorname{ht}(\lambda) - 2=2\ell(t_{-\lambda})-2$. So we have that the homology
class $[\MV_{ \lambda,v}]$  
is a positive integer multiple of the Schubert class $\xi_{s_i t_{-\lambda}}$.
 Proposition \ref{p:MVmatrix}  allows  one to write 
$[\MV_{ \lambda,v}]$ as a matrix coefficient on $X$ and compare it with Kostant's formula for $q_{\lambda}\si_{s_i}$.  
Finally we see in this way that $q_{\lambda}\si_{s_i}$ is a positive integral multiple of $ \xi_{s_i t_{-\lambda}}$  as
function on $X_B$.
  
Now we can compose this isomorphism $qH^*(G/B)[q_1\inv,\dotsc,q_n\inv]\cong\mathcal O(X_B)\cong H_*(Gr_G)[\xi_{t_\mu}\inv]$ which we have
just seen takes $\sigma_{s_i}$ to a positive integral multiple of   $ \xi_{s_i t_{-\lambda}}\xi_{t_{-\lambda}}\inv$,
with the isomorphism $ H_*(Gr_G)[\xi_{t_\mu}\inv]\to qH^*(G/B)[q_1\inv,\dotsc,q_n\inv]$ defined by \eqref{e:PetIsoSchub} going the other way.
This way we get a map  $qH^*(G/B)[q_1\inv,\dotsc,q_n\inv]\to qH^*(G/B)[q_1\inv,\dotsc,q_n\inv]$ which takes every $q_i$ to $q_i$ but
$\sigma_{s_i}$ to a positive integral multiple of $\sigma_{s_i}$. However the images of the $\sigma_{s_i}$ still need to 
obey the unique quadratic relation in the quantum cohomology ring, which identifies a quadratic form in the $\sigma_{s_i}$'s
with a linear form in the $q_i$'s. The $\sigma_{s_i}$'s cannot be rescaled by positive integer multiples and this relation
still hold, unless all of the integer factors are $1$, which means that $\sigma_{s_i}$ must go to $\sigma_{s_i}$. 

Since $qH^*(G/B)[q_1\inv,\dotsc,q_n\inv]$ is generated by
the  $\sigma_{s_i}$ as ring over $\C[q_1^{\pm 1},\dotsc q_n^{\pm 1}]$ the map  $qH^*(G/B)[q_1\inv,\dotsc,q_n\inv]\to qH^*(G/B)[q_1\inv,\dotsc,q_n\inv]$ 
considered above is the identity. Therefore \eqref{e:PetIsoSchub}  is the inverse to the map 
$qH^*(G/B)[q_1\inv,\dotsc,q_n]\cong\mathcal O(X_B)\cong H_*(Gr_G)[\xi_{t_\mu}\inv]$, and we are done.

 \end{rem}

\section{Main results}
In the rest of the paper we will be working with the $\R$-structures on all our main objects.  Since $X, \Y, H_*(\Gr_G), qH^*(G/P)$ are in fact all defined over $\Z$ there is no problem with this.  All our notations for positivity and nonnegativity refer to $\R$-points.  We now define

\begin{enumerate}
\item[(AP)] The subset
$$
X^{\af}_{>0} = \{x \in X \mid \xi_w(x) > 0 \,\text{for all}\, w \in
W_\af^-\}
$$
of {\it affine Schubert positive} elements.
\item[(MVP)] The subset
$$
X^{\MV}_{>0} = \{x \in X \mid [\MV_{\lambda,v}](x) > 0 \,\text{for
all $\MV$-cycles}\}
$$
of {\it MV-positive} elements.
\item[(TP)] The {\it totally positive subset} defined by Lusztig's theory, $$X_{>0}:=X\cap U^\vee_{-,>0}.$$

\item[(QP)]
The subset
$$
X^{\q}_{>0}:=\{x\in X_B\ | \sigma^w_B(x)>0\text{ all $w\in W$}\}
$$
of {\it quantum Schubert positive} elements defined in terms of the
quantum Schubert basis and Peterson's isomorphism
$qH^*(G/B)[q_1\inv,\dotsc, q_n\inv]\cong \mathcal O(X_B)$.
\end{enumerate}

Define the totally nonnegative part $X_{\geq 0} = X \cap U^\vee_{-,\geq 0}$ of $X$, and the affine Schubert nonnegative part $X^\af_{\geq 0}$ of $X$ as the set of points $x \in X$ such that $\xi_w(x) \geq 0$ for every affine Schubert class $\xi_w$.
We can now state our first main theorem.

\begin{thm}\label{thm:three}
The first three notions of positivity in $X$ agree:  we
have $X_{>0} = X^{\af}_{>0} = X^\MV_{>0}$, and for the fourth we have $X_{>0} \subset X^{\q}_{>0}$.  Furthermore, we have $X_{\geq 0} = \overline{X_{> 0}}  = \overline{X^{\af}_{> 0}} = X^{\af}_{\geq 0}$.
\end{thm}

Therefore the first three notions of positivity are equivalent, and
the fourth is at worst weaker.  We note that by \eqref{e:PetIsoLoopHom} affine Schubert positivity is
equivalent to quantum Schubert positivity with additional positivity
of the quantum parameters.  Therefore $X^{\af}_{> 0} \subset X^{\q}_{> 0}$ is immediate.

\begin{conjecture}\label{c:QP}
We have $X_{> 0} = X^{\q}_{>0}$.
\end{conjecture}

Conjecture \ref{c:QP} was shown in \cite{Rie:QCohPFl} for type $A$.  We will verify it in Appendix \ref{s:proof1} for type $C$.

By the isomorphism
$$
\C
[\Y^*_{P}]\cong qH^*(G/P)[(q^P_1)\inv,\dotsc,(q^P_k)\inv]
$$
from Theorem~\ref{t:Pet} combined with  $X_{P}\cong \Y^*_P$ we have a morphism
$$
\pi^P=(q_1^P,\dotsc, q_k^P): X_P \to (\C^*)^k.
$$
Let $X_{P,>0}:=X_P\cap X_{\ge 0}$. In particular
$X_{B,>0}=X_{>0}$. 
\begin{thm}\label{t:main}\
\begin{enumerate}
\item
$\pi^P$ restricts to a bijection
\begin{equation*}
\pi^P_{>0}:X_{P,>0}\to\R^k_{>0}.
\end{equation*}
\item
$X_{P,>0}$ lies in the smooth locus of $X_{P}$, and the map
$\pi^P$
is etale on $X_{P,>0}$.
\item
The maps $\pi^P_{>0}$ glue to give a homeomorphism
\begin{equation*}
\Delta_{\ge 0}: X_{\ge 0}\To \R_{\ge 0}^{n}.
\end{equation*}
\end{enumerate}
\end{thm}


\section{One direction of Theorem \ref{thm:three}}
The main goal of this section is to show that
$$
X^{\af}_{>0} \subseteq X^\MV_{>0} \subseteq X_{>0}.
$$

\begin{lem}\label{l:AFtoMV} If $x\in X$ is affine Schubert positive, then it is MV-positive.
\end{lem}
\begin{proof}
The main result of Kumar and Nori \cite{KuNo:pos}, applied to
$\Gr_G$, shows that every effective cycle in $\Gr_G$ is homologous to
a positive sum of Schubert cycles.  It follows that the fundamental class $[\MV_{\lambda,v}]$ of an MV-cycle is a positive linear combination of the Schubert classes $\xi_w$.
\end{proof}

\begin{lem}\label{l:MVtoTP} Suppose $ X^{\MV}_{>0}\cap X_{>0}\ne\emptyset$. Then $ X^{\MV}_{>0}\subseteq X_{>0}$.
\end{lem}

\begin{proof} 
Suppose  $V_{\lambda}$ is a representation of $G^\vee$ with highest weight $\lambda$ and  $ \mu$ is a weight of $V_{\lambda}$ with
one-dimensional weight space. 

Let $[\MV_{\lambda,\mu}]$ be the $MV$-cycle representing a weight vector $v$ with weight $\mu$ in $V_\lambda$ under the  isomorphism \eqref{E:IC}.  Then for $x \in X$, we have by Proposition \ref{p:MVmatrix}, $[MV_{\la,\mu}](x)  = \ip{x,[MV_{\la,\mu}]} = \ip{x \cdot v, v_\la^-}= \ip{ v, x^T \cdot v_\la^-}$. 
Note that we are really thinking of $V_\lambda$ as a lowest weight representation by fixing the lowest weight vector $v_\la^-$ (of weight $w_0\la$).

Now suppose that that there is a vector $v'$ in the weight space $V_\lambda(\mu)$ satisfying $\ip{v', y \cdot v_{\la}^-} > 0$ for all $y \in U^+_{>0}$.
Since the weight space $V_\lambda(\mu)$ is one-dimensional it follows that the $MV$-basis element $v = c_{\lambda,\mu} \, v'$ for
a scalar $c_{\lambda,\mu}$.  Choose $x_0\in X^{\MV}_{>0}\cap X_{>0}$. Such an $x_0$ exists by our assumption.  Then we see that
the scalar is positive,
$$
c_{\lambda,\mu}=\frac{\left<x_0\cdot v,
 v_{\la}^-\right>}{\left<x_0\cdot v',
 v_{\la}^-\right>}=\frac{[MV_{\la,\mu}](x_0)}{\left<v',x_0^T\cdot
 v_{\la}^-\right>}>0.
$$
Now suppose $x\in X^{\MV}_{>0} $ is an arbitrary element.
Then
$$
\left<v',x^T\cdot v_{\la}^-\right> = \frac
1{c_{\lambda,\mu} }\ip{v,x^T\cdot  v_{\la}^-} = \frac
1{c_{\lambda,\mu} }[\MV_{\lambda,\mu}](x)
>0.
$$
By Proposition \ref{P:genBZ} (applied with `positive Borel' $B^+$ taken to be $B^\vee_-$) , this implies that $x^T$ is totally positive. 
Clearly then $x$ is totally positive.
\end{proof}

\begin{lem}\label{l:fcheck}
The principal nilpotent $f^\vee=\sum f^\vee_i$ goes to a
 positive multiple of the affine Schubert class $\xi^{s_0}\in H^2(\Gr_G)$
 under the isomorphism $\mathcal U((\mathfrak g^{\vee})^{F})\cong H^*(\Gr_G) $.
\end{lem}

\begin{proof}
Since $H^2(\Gr_G)$ is $1$-dimensional we know that $f^\vee=c\xi^{s_0}$
under the identification  $\mathcal U((\mathfrak g^{\vee})^{F})\cong H^*(\Gr_G) $.
We want to show that $c$ is positive. 

 Consider the exponential $\exp(f^\vee)=\exp(c\xi^{s_0})$ as an element
of the completion $\widehat{ H^*(\Gr_G)}$, and 
choose $\lambda$ such that $s_it_\lambda\in W^-_{\af}$. 
If we evaluate the (localized) homology class
$\xi_{s_it_\lambda}\xi_{t_\lambda}\inv$ on this element we obtain
\begin{align*}
\xi_{s_it_\lambda}\xi_{t_\lambda}\inv(\exp(f^\vee))&=\xi_{s_it_\lambda}\xi_{t_\lambda}\inv(\exp(c\xi^{s_0}))
\\
&=\xi_{s_it_\lambda}\xi_{t_\lambda}\inv(1+c\xi^{s_0} +
\frac{c^2}{2!}(a_2) + \frac{c^3}{3!}(a_3) + \cdots)
\\
&=
\frac{c^{k-1}}{(k-1)!}\frac{k!}{c^{k}}\xi_{s_it_\lambda}(a_{k-1})\xi_{t_\lambda}(a_{k})\inv
\\
&
=\frac{k}{c}\ \xi_{s_it_\lambda}(a_{k-1})\xi_{t_\lambda}(a_{k})\inv,
\end{align*}
where $\ell(t_\lambda) = k = \ell(s_it_\lambda) + 1$ and the $a_i
\in H^{2i}(\Gr_G)$ are positive linear combinations of Schubert
classes, since all cup product Schubert structure constants of
$H^*(\Gr_G)$ are positive \cite{KuNo:pos}. Therefore $c$ is positive
if and only if
$\xi_{s_it_\lambda}\xi_{t_\lambda}\inv(\exp(f^\vee))$ is positive.

We now compute $\xi_{s_it_\lambda}\xi_{t_\lambda}\inv(\exp(f^\vee))$  in
 a different way. Under
Peterson's isomorphism \eqref{e:PetIsoLoopHom} 
\begin{equation}\label{e:ExamplePetIso}
\xi_{s_it_\lambda}\xi_{t_\lambda}\inv\mapsto \sigma_{s_i}.
\end{equation}
Consider  the element of $X_{>0}$, the totally positive part of $X_B$, given by $\exp(f^\vee)$. 
Using \eqref{e:ExamplePetIso} and identifying both $H_*(\Gr_G)$ and $qH^*(G/B)[q_1\inv,\dotsc, q_n\inv ]$ with
$\C[X_B]$ as in \eqref{e:PetIsoLoopHom}, we can evaluate 
\begin{equation}\label{e:qCohEval}
\xi_{s_it_\lambda}\xi_{t_\lambda}\inv(\exp(f^\vee))=\sigma_{s_i}
(\exp(f^\vee)).
\end{equation}
The right hand side here is a quotient of two `chamber minors' of $\exp(f^\vee)$ by
Kostant's formula \cite[Proposition 33]{Kos:QCoh}. Therefore the total positivity of
$\exp(f^\vee)$ implies
\begin{equation}\label{e:xipos}
\xi_{s_it_\lambda}\xi_{t_\lambda}\inv(\exp(f^\vee))> 0.
\end{equation}

\end{proof}

\begin{lem}\label{l:SchubPosx} The element $x=\exp(f^\vee)$ lies in $X^{\af}_{>0}$.
\end{lem}
\begin{proof}
A Schubert class $\xi_w$ in $H_*(\Gr_G)$ can be evaluated against
$\exp(f^\vee)$ by viewing $f^\vee$ as element of $H^2(\Gr_G)$,
expanding $\exp(f^\vee)$ as a power series and pairing $\xi_w$ with
each summand. But Lemma~\ref{l:fcheck} implies that $\exp(f^\vee)$
expands as a positive linear combination of Schubert 
classes. This implies that $\xi_w(\exp(f^\vee))>0$ for all $w\in
W_{\af}^-$.
\end{proof}

\begin{lem} We have $X_{>0}^{\MV}\subseteq X_{>0}$ and
$X_{\ge 0}^{\MV}\subseteq X_{\ge 0}$.
\end{lem}
\begin{proof}
In Lemma~\ref{l:SchubPosx} we found a totally positive point $x \in X_{>0}$,
namely $x=\exp(f^\vee)$, which is also affine Schubert positive.
Since affine Schubert positive implies MV-positive, by Lemma~\ref{l:AFtoMV}, this means that $x\in X^{\MV}_{>0}\cap X_{>0}$.
Now Lemma~\ref{l:MVtoTP} implies that $ X_{>0}^{\MV}\subseteq
X_{>0}$. The second inclusion is an immediate consequence.
\end{proof}

\begin{cor}\label{c:AFtoTN} We have 
$X^{\af}_{>0} \subseteq X^\MV_{>0} \subseteq X_{>0}$ and $X^{\af}_{\geq 0} \subseteq X^\MV_{\geq 0} \subseteq X_{\geq 0}$.
\end{cor}

\section{Parametrizing the affine Schubert-positive part of $X_P$}
Let $X_{P,>0}^{\af} = X^{\af}_{\geq 0} \cap X_P$ denote the points $x \in X_P$ such that $\xi_w(x) \geq 0$ for every affine Schubert class $\xi_w$.  
First we note that $\pi^P$ takes
values in $\R_{>0}^k$ on $X_{P,>0}^{\af}$.  Indeed, by definition $q^P_i(x) \neq 0$ for $x \in X_P$, and expressing
the quantum parameters $q^P_i$ in terms of affine Schubert classes
$\xi_{t_\lambda}$ using
Theorem~\ref{t:PetIsoLoop}(2), it follows that $x\in X_{P,>0}^{\af}$ has $q^P_i(x)>0$ for all $i\in I^P$.  
It follows from the following result that we also have $X^{\af}_{B, >0} = X^{\af}_{>0}$.

\begin{lem}
Suppose $x \in X_{P,>0}^{\af}$.  Then $\sigma_w^P(x) > 0$ for all $w\in W^P$.
\end{lem}
\begin{proof}
It follows from Theorem~\ref{t:PetIsoLoop}(2) and the definitions that $\sigma_w^P(x) \geq 0$.  Suppose $\sigma_w^P(x) = 0$.  Let $r_\theta$ denote the reflection in the longest root, and $\pi_P(r_\theta) \in W^P$ be the corresponding minimal length parabolic coset representative.
Proposition 11.2 of \cite{LaSh:QH} states that
$$
\sigma^P_{\pi_P(r_\theta)}\, \sigma_w^P =q_{\eta_P(\theta^\vee - w^{-1}\theta^\vee)} \sigma_{\pi_P(r_\theta w)}^P +  q_{\eta_P(\theta^\vee)} \sum_{s_i w < w} a_i^\vee\, \sigma_{s_i w}^P.
$$
We refer the reader to Appendix \ref{s:proof1} and \cite{LaSh:QH} for the notation used here.  Applying this repeatedly, we see that for large $\ell$, the product $(\sigma^P_{\pi_P(r_\theta)})^{\ell} \, \sigma_w^P$ is a (positive) combination of quantum Schubert classes which includes a monomial in the $q_i^P$.  This contradicts $q_i^P(x) > 0$ for each $i$.
\end{proof}

The remainder of this section will be devoted to the proof of the following proposition. 

\begin {prop} The map $\pi^P_{>0}:X^{\af}_{P,>0}\To \R_{>0}^{k}$ is
bijective. 
\end{prop}

 We follow the proof in type $A$ given in \cite{Rie:QCohPFl}, shortening 
somewhat the proof of our Lemma \ref{l:indec}  below (Lemma~9.3 in~\cite{Rie:QCohPFl}), by using
a result of Fulton and Woodward
\cite{FuWo:SchubProds}: the quantum
product of Schubert classes is always nonzero.

 Fix a point $Q\in (\R_{>0})^{k}$ and
consider its fiber under $\pi=\pi^P$. Let us define
\begin{equation*}
R_Q:=qH^*(G/P)/(q^P_1-Q_1,\dotsc, q_k^P-Q_k).
\end{equation*}
This is the (possibly non-reduced) coordinate ring of $\pi\inv(Q)$.
Note that $R_Q$ is a $|W^P|$-dimensional algebra with basis given
by the (image of the) Schubert basis. We will use the same
notation $\si^P_w$ for the image of a Schubert basis element
from $qH^*(G/P)$ in  the quotient $R_Q$.  The proof of the following result from \cite{Rie:QCohPFl} holds in our situation verbatim.

\begin{lem}[{\cite[Lemma~9.2]{Rie:QCohPFl}}]\label{l:EV}
Suppose $\mu\in R_Q$ is a nonzero simultaneous eigenvector for all
linear operators $R_Q\to R_Q$ which are defined by multiplication by
elements in $R_Q$. Then there exists a point $p\in\pi\inv(Q)$ such that (up to
a scalar factor)
\begin{equation*}
\mu=\sum_{w\in W^P}\si^P_w(p)\,\si^P_{PD(w)}.
\end{equation*}\qed
\end{lem}

Set
\begin{equation*}
\sigma:=\sum_{w\in W^P}\si_w^P\in R_Q.
\end{equation*}
Suppose the multiplication operator on $R_Q$ defined by
multiplication by $\si$ is given by the matrix
$M_\si=(m_{v,w})_{v,w\in W^P}$ with respect to the Schubert
basis. That is,
\begin{equation*}
\si\cdot\si_v^P=\sum_{w\in W^P} m_{v,w}\si^P_w.
\end{equation*}
Then since $Q\in \R_{>0}^k$ and by positivity of the structure
constants it follows that $M_\si$ is a nonnegative matrix.

\begin{lem} [{\cite[Lemma~9.3]{Rie:QCohPFl}}]\label{l:indec}
$M_\si$ is an indecomposable matrix.
\end{lem}
\begin{proof}
Suppose indirectly that the matrix $M_\si$ is
reducible. Then there exists a nonempty, proper subset $V\subset
W^P$ such that the span of $\{\si_v\ | \ v\in V \}$ in $R_Q$ is
invariant under $M_\si$. We will derive a contradiction to this
statement.

First let us show that $1\in V$. Suppose not. Since $V\ne\emptyset$
we have a $v\ne 1$ in $V$. Since $1\notin V$, the coefficient of
$\si_{1}$ in $\si_{w}\cdot\si_{v}$ must be zero for all $w\in
W^P$, or equivalently
\begin{equation}\label{e:zero}
\left\langle\,\si_w\cdot\si_{v}\cdot \si_{w_0^P}\,\right\rangle_Q=0
\end{equation}
for all $w\in W^P$. Here by the bracket $\left\langle\,\quad \,\right\rangle_Q$ we mean 
$\left\langle\,\quad \,\right\rangle_{\mathbf q}$ evaluated at $Q$. 
But this \eqref{e:zero} implies
$\langle\si_w\cdot\si_{v}\cdot \si_{w_0^P}\rangle_\mathbf
q=0$, since the latter is a nonnegative polynomial in the
$q_i^P$'s which evaluated at $Q\in \R_{>0}^k$ equals $0$.
Therefore $\si_{v}\cdot\si_{w_0^P}=0$ in $qH^*(G/P)$, by
quantum Poincar\'e duality. This leads to a contradiction,
since by work of W.~Fulton and C.~Woodward
\cite{FuWo:SchubProds} no two Schubert classes in $qH^*(G/P)$ ever
multiply to zero.

So $V$ must contain $1$. Since $V$ is a proper subset of $W^P$ we
can find some $w\notin V$. In particular, $w\ne 1$. It is a
straightforward exercise that given $1\ne w\in W^P$ there exists
$\alpha\in \Delta_+^P$ and $v\in W^P$ such that
\begin{equation*}
w=v r_\alpha, \quad\text{and}\quad \ell(w)=\ell(v)+1.
\end{equation*}
Now $\alpha\in \Delta_+^P$ means there exists $i\in I^P$ such
that $\langle \alpha,\omega_{i}^\vee\rangle\ne 0$. And hence by the
(classical) Chevalley Formula we have that
$\si_{s_{i}}\cdot\si_v$ has $\si_w$ as a summand. But if
$w\notin V$ this implies that also $v\notin V$, since $\si\cdot
\si_v$ would have summand $\si_{s_{i}}\cdot\si_v$ which has
summand $\si_w$. Note that there are no cancellations with other
terms by positivity of the structure constants.

By this process we can find ever smaller elements of $W^P$ which
do not lie in $V$ until we end up with the identity element, so a
contradiction.
\end{proof}

Given the indecomposable nonnegative matrix $M_\si $,
then by Perron-Frobenius theory (see e.g. \cite{Minc:NonnegMat}
Section 1.4) we know the following. \vskip .3cm
\parbox[c]{12cm}{
The matrix $M_\si$ has a positive eigenvector $\mu$ which is
unique up to scalar (positive meaning it has positive coefficients
with respect to the standard basis). Its eigenvalue, called the
Perron-Frobenius eigenvalue, is positive, has maximal absolute
value among all eigenvalues of $M_{\sigma}$, and has algebraic
multiplicity $1$. The eigenvector $\mu$ is unique even in the
stronger sense that any nonnegative eigenvector of $M_\sigma$ is a
multiple of $\mu$.} \vskip .3cm

Suppose $\mu$ is this eigenvector chosen normalized such that
$\left<\mu\right>_Q=1$. Then since the eigenspace containing $\mu$
is $1$--dimensional, it follows that $\mu$ is joint eigenvector
for all multiplication operators of $R_Q$. Therefore by
Lemma~\ref{l:EV} there exists a $p_0\in \pi\inv(Q)$ such that
\begin{equation*}
\mu=\sum_{w\in W^P}\si^P_w(p_0)\, \si^P_{PD(w)}.
\end{equation*}
Positivity of $\mu$ implies that $\si^P_w(p_0)\in\R_{>0}$ for all
$w\in W^P$. Of course all of the
$q_i(p_0)=Q_i$ are positive too. Hence $p_0\in X_{P,>0}^{\af}$. Also the point $p_0$
in the fiber over $Q$  with the property that all $\si^P_w(p_0)$ are positive is unique.  
Therefore 
\begin{equation}\label{e:homeo}
X^{\af}_{P,>0}\To \R_{>0}^{k}
\end{equation}
is a bijection.

\section{Proof of Theorem \ref{t:main}.(2)}\label{s:proof2}
We establish Theorem \ref{t:main}(2) for $X^{\af}_{P,>0}$ instead of $X_{P,>0}$.  In Proposition \ref{p:closures}, we will we establish the equality $X^{\af}_{P,>0} = X_{P,>0}$.

  Since $qH^*(G/P)$ is free over $\C[q_1^P,\ldots,q_k^P]$, it follows that $\pi^P$ is flat.
Let $Q=\pi^P(p_0)$. Let $R=qH^*(G/P)$ and $I\subset R$ the ideal
$(q_1-Q_1,\dotsc, q_k-Q_k)$.
The Artinian ring $R_Q=R/I$ is isomorphic to the sum of local
rings $R_Q\cong\bigoplus_{x\in (\pi^P)\inv(Q)} R_x/IR_x$. And for
$x=p_0$ the local ring $R_{p_0}/IR_{p_0}$ corresponds in $R_Q$ to
the Perron--Frobenius eigenspace of the multiplication operator
$M_\sigma$ from the above proof. Since this is a one-dimensional
eigenspace (with algebraic multiplicity one) we have that
$\dim(R_{p_0}/IR_{p_0})=1$.  It follows that the map $\pi^P$ is unramified at the point $p_0$.  Thus, for example by \cite[Ex.III.10.3]{Hartshorne}, $\pi^P$ is etale at $p_0$.  Since $(\C^*)^k$ is smooth, it follows that $X_P$ is smooth at $p_0$.

\section{Proofs of Theorem \ref{thm:three} and Theorem \ref{t:main}(1)}
\begin{lem}\label{l:semi}
$X_{\geq 0}$ and $X^\af_{\geq 0}$ are closed subsemigroups of $X$.
\end{lem}
\begin{proof}
For $X_{\geq 0}$ this follows from the fact that $(U^\vee_-)_{> 0}$ is a subsemigroup of $U^\vee_-$, and $X \subset U^\vee_-$ is a subgroup.  
For $X^\af_{\geq 0}$, closed-ness follows from the definition.  Suppose $x, y \in X^\af_{\geq 0}$.
 Then for any affine Schubert class $\xi_w$, we have $$\xi_w(xy) =
\Delta(\xi_w)(x \otimes y) = \sum_{v,u}
c_{v,u}^w \xi_v(x) \otimes \xi_u(y) \geq 0$$ where $\Delta$ denotes the coproduct of $H_*(\Gr_G)$, and $c_{v,u}^w \geq 0$ are  nonnegative integers \cite{KuNo:pos}.  Thus $xy \in X^\af_{\geq 0}$.
\end{proof} 

 The first statement of Theorem \ref{thm:three} follows from Corollary \ref{c:AFtoTN} and the following proposition.
\begin{prop} \label{p:connected} The totally positive part and the affine Schubert positive part
of $X$ agree,
\begin{equation*}
X^{\af}_{>0}=X_{>0}.
\end{equation*}
\end{prop}

\begin{proof}
Our proof is identical to the proof of Proposition~12.2 from \cite{Rie:QCohPFl}.

By \cite[Section 5]{Kos:QCoh} or combining Theorem \ref{e:PetIsoLoopHom} with \cite[Proposition 1.9]{Gin:GS}, we see that each $q_i$ is a ratio of `chamber minors' and so $\pi^B$ takes positive values on $X_{>0}$.  By  Corollary \ref{c:AFtoTN} we have the following commutative diagram
\begin{equation*}
\begin{matrix}X^{\af}_{>0}&\hookrightarrow &X_{>0}\\
\qquad   \searrow & & \swarrow\qquad \\
      &\R_{>0}^{n}&
\end{matrix}
\end{equation*}
where the top row is clearly an open inclusion and the maps going
down are restrictions of $\pi^B$. By \eqref{e:homeo} and Section \ref{s:proof2}, the left hand map to $\R_{>0}^{n}$ is a homeomorphism. It follows from this and elementary point set
topology that $X^{\af}_{>0}$ must be closed inside $X_{>0}$. So
it suffices to show that $X_{>0}$ is connected.

For an arbitrary element $u\in X$ and $t\in \R$, let
\begin{equation}\label{e:ut}
u_t:=t^{-\rho} u t^{\rho},
\end{equation}
where $t\mapsto t^\rho$ is the one-parameter subgroup of
$T^\vee$ corresponding to the coroot $\rho$ (a coroot relative to $G^\vee$).
Then $u_0=e$ and $u_1=u$, and if $u\in X_{>0}$, then so is $u_t$
for all positive $t$.

Let $u, u'\in X_{>0}$ be two arbitrary points. Consider the
paths
\begin{eqnarray*}
\gamma\ : {[0,1]\to X_{>0} ~,}& \gamma(t) =u u'_t\ \\
\gamma' : {[0,1]\to X_{>0} ~,}& \,\gamma'(t)=u_t u'.
\end{eqnarray*}
Note that these paths lie entirely in $X_{>0}$ since $X_{>0}$
is a semigroup (Lemma \ref{l:semi}). Since $\gamma$ and $\gamma'$ connect $u$
and $u'$, respectively, to $uu'$, it follows that $u$ and $u'$
lie in the same connected component of $X_{>0}$, and we are
done.
\end{proof}
%

The second statement of Theorem \ref{thm:three} and Theorem \ref{t:main}(1) follow from:
\begin{prop}\label{p:closures}
We have $\overline{X_{> 0}} = X_{\geq 0} =X \cap U^\vee_{-,\geq 0}$.  We have $\overline{X^\af_{>0}}=X^\af_{\geq 0}$.  Thus 
$$X_{P,>0}=X^{\af}_{P,>0}.$$
\end{prop}
\begin{proof}
Suppose $x\in X_{\geq 0}$.  Then for any $u\in X_{>0}$, we have $u_t \in X_{>0}$ for all positive $t$, 
where $u_t$ is defined in \eqref{e:ut}.  The
curve $t\mapsto x(t)=xu_t$ starts at $x(0)=x$ and lies in
$X_{>0}$ for all $t>0$. Therefore $x\in \overline{X_{>0}}$ as desired.

The same proof holds for $X^\af_{>0}$, using Lemma \ref{l:semi} and the fact that $u_t \in X_{>0} = X^\af_{>0}$ (Proposition \ref{p:connected}).
\end{proof}

%

\section{Proof of Theorem \ref{t:main}(3)}
To define $\Delta_{\geq 0}$, we set $\Delta_i = \xi_{t_{m_i \omega_i^\vee}}$, where $m_i$ is chosen so that $m_i \omega_i^\vee \in Q^\vee$.  Then $\Delta_{\geq 0} = (\Delta_1,\ldots,\Delta_{n})$.

It follows from the explicit description \cite{LaSh:QH} of $\pi_P(t_\la)$ and $\eta_P(\la)$ of Theorem \ref{t:PetIsoLoop} that for each $i$, some power of $q^P_i$ is equal to $\xi_\la \xi_\mu^{-1}$ on $X_P$, for certain $\la, \mu \in Q^\vee$.  Furthermore, the map $$\pi^P_{>0}=(q^P_1,\dotsc, q^P_k):X_{P,>0}\to \R^{k}_{>0}$$ is related to  the map $$\Delta^P_{>0} = (\Delta_{i_1},\Delta_{i_2},\ldots,\Delta_{i_k}): X_{P,>0}\to \R^{k}_{>0}$$ by a homeomorphism of $\R^k_{>0}$, where $I^P = \{i_1,i_2,\ldots,i_k\}$.
But $X_{\ge 0}=\bigsqcup X_{P,>0}$, so we have that
\begin{equation*}
\Delta_{\ge 0}: X_{\ge 0}\To \R_{\ge 0}^{n}
\end{equation*}
is bijective. So $\Delta_{\ge 0}$ is continuous and bijective.
Since $\Delta$ is finite it follows that it is closed, that is, takes closed sets to closed
sets. (This holds true also in the Euclidean topology, since the 
preimage of a bounded set under a finite map must be bounded, compare
\cite[Section~5.3]{Shafarevich:AG1}). Since $X_{\ge 0}$ 
is closed in $X$ the restriction $\Delta_{\ge 0}$ of $\Delta$ to $X_{\ge 0}$ is also 
closed. Therefore $\Delta_{\ge 0}\inv$ is continuous.


\section{Proof of Proposition \ref{P:genBZ}}\label{s:genBZ}
It suffices to prove the Proposition for $G$ of adjoint type.   Call a dominant weight $\lambda$ {\it allowable} if it is 
a character of the maximal torus of adjoint type $G$.

We note that the tensor product $V = V_\la \otimes V_\mu$ of two irreducible representations inherits a {\it tensor Shapovalov form} $\ip{\cdot,\cdot}$ defined by $\ip{v \otimes w, v' \otimes w'} = \ip{v,v'}\ip{w,w'}$.  This is again a positive-definite non-degenerate symmetric form on $V_{\la,\R} \otimes V_{\mu,\R}$ satisfying \eqref{E:adjoint}.  It follows from \eqref{E:adjoint} that if $V_\nu, V_{\rho} \subset V$ are irreducible subrepresentations, and $\nu \neq \rho$ then $\ip{v,v'} = 0$ for $v \in V_\nu$ and $v' \in V_\rho$. Thus if the highest-weight representation $V_\nu$ occurs in $V$ with multiplicity one, the restriction of $\ip{\cdot,\cdot}$ from $V$ to $V_\nu$ must be a positive-definite non-degenerate symmetric bilinear form satisfying $\eqref{E:adjoint}$, and thus must be a multiple of the Shapovalov form.  By scaling the inclusion $V_\nu \subset V$, we shall always assume that the restricted form is the Shapovalov form.  The above comments extend to the case of $n$-fold tensor products.

\subsection{Type $A_n$}
We shall establish the criterion used in Proposition \ref{P:BZ}.  First suppose $n$ is even.  Let $V_{\omega_i}$ be a fundamental representation, and let $v_{\omega_i}^+ \in V_{\omega_i}$ be the highest weight vector, and $v = \dot w \cdot v_{\omega_i}^+$ an extremal weight vector.  The weight space with weight $\dot w \cdot (n+1)\omega_i$ is extremal (and one-dimensional) in $V_{(n+1) \omega_i}$, and $V_{(n+1) \omega_i}$ is an irreducible representation for $PSL_{n+1}(\C)$.  Thus for $y$ as in Proposition~\ref{P:BZ},
$$
\ip{v,y \cdot v^+_{\omega_i}}^{n+1} = 
\ip{v^{\otimes {(n+1)}}, y\cdot (v^+_{\omega_i})^{\otimes {(n+1)}}} > 0.
$$
Since $n$ is even, this implies that $\ip{v,y \cdot v_{\omega_i}^+} > 0$.

For odd $n$, let us fix $w \in W$, and consider the set of signs $a_i =
\sign(\ip{\dot w \cdot v_{\omega_i}, y \cdot v_{\omega_i}})$.  We want to prove that the $a_i$ are
all $+1$.  Note that a sum of (not necessarily distinct) fundamental weights, 
$\omega_{i_1}+ \cdots +\omega_{i_k}$, is allowable precisely if it is trivial on the center of $SL_{n+1}$, that is if $i_1 + \cdots + i_k$ is 
divisible by $n+1$. Let $(i_1,i_2,\ldots,i_k)$ be such a sequence of indices, for which 
$\omega_{i_1}+ \cdots +\omega_{i_k}$ is allowable. Then the
weight $w(\omega_{i_1}+ \cdots +\omega_{i_k})$ is an extremal weight of 
the representation $V_{\omega_{i_1}+ \cdots +\omega_{i_k}}$ of $PSL_{n+1}(\C)$, and we have 
\begin{multline*}
\ip{\dot w\cdot v^+_{\omega_{i_1}},y \cdot v_{\omega_{i_1}}^+}
\ip{\dot w\cdot v^+_{\omega_{i_2}},y \cdot v_{\omega_{i_2}}^+}\cdots
\ip{\dot w\cdot v^+_{\omega_{i_k}},y \cdot v^+_{\omega_{i_k}}} \\= 
\ip{\dot w\cdot(v^+_{\omega_{i_1}}\otimes \dotsc\otimes v^+_{\omega_{i_k}}), y\cdot (v^+_{\omega_{i_1}}\otimes \dotsc\otimes v^+_{\omega_{i_k}})} > 0.
\end{multline*}
Therefore $a_{i_1}\dotsc a_{i_k}=+1$ if $i_1+ \dotsc + i_k= n+1$. In particular $a_{i} a_1^{n+1-i}=+1$, implying that $a_i=+1$ for even $i$,
and $a_i=a_1$ for odd $i$.

We now show that $a_1 = +1$.  Let $V = V_{\omega_1}=\C^{n+1}$ with standard basis $\{v_1,\dotsc, v_{n+1}\}$, and let $Z= V^{\otimes (n+1)}$. 
If we take $v^+_{\omega_1}=v_1$ then the Shapovalov form on $V$ is the standard symmetric bilinear form given by $\ip{v_i,v_j}=\delta_{i,j}$.  
Let us consider $U =V_{(n+1)\omega_1} = \Sym^{n+1}(V)$, which occurs with multiplicity $1$ in $Z$ and has standard basis $\{v_{i_1}\dotsc v_{i_{n+1}} \ |\ 1\le i_1\le i_2\le \dotsc\le i_{n+1}\le n+1\ \}$ of symmetrized tensors,
$$
v_{i_1}\dotsc v_{i_{n+1}}=\frac{1}{(n+1)!}\sum_{\sigma\in S_{n+1}} v_{\sigma(i_1)}\otimes v_{\sigma(i_2)}\otimes\cdots\otimes v_{\sigma(i_{n+1})}.
$$ 
These are clearly orthonormal for the tensor Shapovalov form restricted to $U$, which is the Shapovalov form $\ip{\ ,\ }_U$ of $U$.
We have $\dot w\cdot  v_1=v_k$ for some $k$. Consider the vector $z=v_1^n v_k \in U$ which has weight $\dot w \cdot \omega_1 + n\omega_1$. Clearly $\ip{z,x \cdot v^+_{(n+1)\omega_1}}_U=\ip{v_1^n v_k,x \cdot v_1^{n+1}}_U >0$ for all totally positive $x \in U^-_{>0}$, and $z$ lies in a $1$-dimensional weight space of $U$. 
Therefore our assumptions imply that
$$
0 < \ip{z, y \cdot v^+_{(n+1)\omega_1}}_U=
\ip{v_1^{n}v_k, y \cdot v_1^{n+1}}_Z=
\ip{v_1, y \cdot v_1}^n \ip{v_k, y \cdot v_1}=\ip{v_k, y \cdot v_1}.
$$
Since $\ip{v_k, y \cdot v_1}_U=\ip{\dot w\cdot v_{\omega_1}^+, y \cdot v_{\omega_1}^+}_U$ this says precisely that $a_1=+1$.


\subsection{Type $B_n$}
The approach we use for the other Dynkin types can also be applied in this case, but we shall proceed using a different approach.
The adjoint group of type $B_n$ is $SO_{2n+1}(\C)$. 
We realize $SO_{2n+1}(\C)$ as subgroup of $SL_{2n+1}(\C)$ 
following Berenstein and Zelevinsky in \cite{BeZe:Chamber} by setting
$$
SO_{2n+1}(\C)=\{ A\in SL_{2n+1}(\C) \ | A J A^t=J\},
$$
for the symmetric bilinear form
$$
J=\begin{pmatrix} &  & & & 1 \\
& & &-1 &\\
& &\iddots & & \\
&-1 & & & \\
1& & & & 
  \end{pmatrix}.
$$ 
Let $\tilde e_i, \tilde f_i$ be the usual Chevalley generators
of $\mathfrak {sl}_{2n+1}$. Then we can take $e_i=\tilde e_i +\tilde e_{2n+1-i}$ and
$f_i=\tilde f_i +\tilde f_{2n+1-i}$ to be Chevalley generators
of $SO_{2n+1}(\C)$, and we have a corresponding pinning. Let $\tilde T$
denote the maximal torus of diagonal matrices in $SL_{2n+1}$ with 
character group $X^*(\tilde T)=\Z\ip{\tilde \ep_1,\dotsc,\tilde \ep_{2n+1} }/ (\sum \tilde \ep_i)$, where
$\tilde\ep_i(t)$ is the $i$-th diagonal entry of $t$.
The maximal torus $T$ of $SO_{2n+1}(\C)$ in this embedding looks like
$$
T=\left\{\left. t= \begin{pmatrix} t_1&  && &  \\
& \ddots & & & &&\\
& &t_n & && &\\
& & & 1 & &&\\
& & & &t_n\inv &&\\
& & & &&\ddots &\\
& & & & &&t_1\inv
  \end{pmatrix}\ \right |  \ t_i\in \C^*
  \right \}.
$$
The restriction of characters from $\tilde T$ to $T$ gives a map
$X^*(\tilde T)\to  X^*(T)$
whose kernel is precisely generated by the characters $\tilde\ep_i+\tilde\ep_{2n-i+2}$
for  $ 1\le i\le n$ and $\tilde\ep_{n+1}$.

By \cite{BeZe:Chamber} the totally
nonnegative part of $SO_{2n+1}(\C)$  
is the intersection of $SO_{2n+1}(\C)$ with the totally nonnegative
part of  $SL_{2n+1}(\C)$.

Consider an element $y \in U^-\subset SO_{2n+1}(\C)$ satisfying 
the condition \eqref{e:allowable}. We want to show that 
$y$ is totally positive as element of $U^-$, or equivalently
by \cite[Corollary~7.2]{BeZe:Chamber}, that $y$ is
totally positive in $\tilde U^-$.

In the Weyl group $\tilde W$ of $SL_{2n+1}$ let $w=(s_1 s_{2n})(s_2 s_{2n-1})\dotsc (s_n s_{n+1}) s_n$.
Then multiplying $w$ with itself $n$ times gives a reduced expression for the longest
element $w_0$ of $\tilde W$. By the Chamber Ansatz of \cite{BeFoZe:TotPos} we can
associate to this reduced expression a set of `chamber minors' which suffice to check
the total positivity of any element of  $\tilde U^-\subset SL_{2n+1}(\C)$. 

The chamber minors can be worked out graphically using the pseudo-line arrangement
for the reduced expression. For $w$ the pseudo-line arrangement is illustrated in Figure~\ref{f:firstw}.
We concatenate $n$ copies of this pseudo-line arrangement together
to get the relevant pseudo-line arrangement $w_0$. To every chamber in the arrangement 
we associate
a set $J\subset \{1,\dotsc, 2n+1\}$, by recording the numbers of the lines running below the chamber.
We order them, so let $j_1<\dotsc<j_k$ be the elements of $J$, and associate a minor to $J$ by setting
$$
\tilde\Delta_{J}(y):=\Delta_{J}(y^T)= \left<y^T\cdot e_{j_1}\wedge\dotsc\wedge e_{j_k}, v_{\omega_k}^+\right>=
\left<e_{j_1}\wedge\dotsc\wedge e_{j_k},y\cdot v_{\omega_k}^+\right>,
$$
where $\Delta_{J}$ is the `chamber minor' as defined in \cite{BeFoZe:TotPos}. Here
$k=|J|$ and $v_{\omega_k}^+=e_1\wedge\cdots\wedge e_k$ is the highest weight 
vector of the irreducible representation $V_{\tilde\omega_k}=\bigwedge^k \C^{2n+1}$ of $SL_{2n+1}(\C)$.  

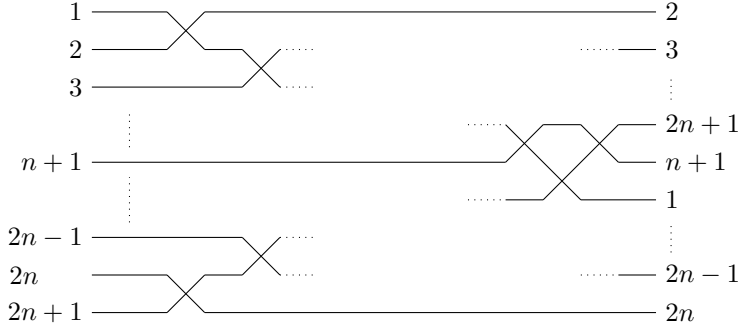
\begin{figure}\label{f:firstw}
\begin{tikzpicture}
\draw (0,3.5)-- (1,3.5);
\node [left] at (0,3.5) {$1$};
\node [left] at (0, 3) {$2$}; 
\draw (0,3)-- (1,3); 
\node [left] at (0, 2.5) {$3$};
\draw (0,2.5)-- (2,2.5);
\draw (2,2.5)--(2.5,3);
\draw (2.5,2.5)--(2,3);
\draw (1.5,3)--(2,3);
\draw [dotted] (.5, 1.7) -- (.5,2.2);
\node [left] at (0, 1.5) {$n+1$\ };
\draw (0,1.5)-- (5.5,1.5);
\draw [dotted] (.5, .7) -- (.5,1.3);
\draw (0,.5)--(2,.5);
\draw (0,0)--(1,0);
\node [left] at (0,0) {$2n\ \ \quad$}; 
\node [left] at (0,.5) {$2n-1$};
\draw (1,3.5) -- (1.5,3);
\draw (1,3) -- (1.5,3.5);
\draw (1,-.5) -- (1.5, 0);
\draw (1,0)-- (1.5, -.5);
\draw (1.5,3.5)--(7.5,3.5);
\draw (1.5,-.5) -- (7.5,-.5);
\draw (0,-.5) -- (1,-.5);
\node [left] at (0,-.5) {$2n+1$}; 
\draw (1.5,0)--(2,0);
\draw (2,0)--(2.5,.5);
\draw (2.5,0)-- (2,.5);
\draw [dotted] (2.5,0)--(3,0);
\draw [dotted]  (2.5,.5)--(3,.5);
\draw [dotted] (2.5,2.5) -- (3,2.5);
\draw [dotted] (2.5,3) -- (3,3);
\draw (5.5,1.5)--(6,2);
\draw (5.5,2)--(6.5,1);
\draw (6.5,1)--(7.5,1);
\draw (6,1)--(7,2);
\draw (6,2)--(6.5,2);
\draw (6.5,2)--(7,1.5);
\draw (7,1.5)--(7.5,1.5);
\draw (5.5,1)--(6,1);
\draw (7,2)--(7.5,2);
\draw [dotted] (6.5,0) -- (7,0);
\draw (7,0) -- (7.5,0);
\draw [dotted] (6.5,3) -- (7,3);
\draw (7,3) -- (7.5,3);
\draw [dotted] (5,2)--(5.5,2);
\draw [dotted] (5,1)--(5.5,1);
\node [right] at (7.5,-.5) {$2n$};
\node [right] at (7.5,0) {$2n-1$};
\node [right] at (7.5,1) {$1$};
\node [right] at (7.5,1.5) {$n+1$};
\node [right] at (7.5,2) {$2n+1$};
\node [right] at (7.5,3) {$3$};
\node [right] at (7.5,3.5) {$2$};
\draw [dotted] (7.7,2.6)--(7.7,2.3);
\draw [dotted] (7.7,.3)--(7.7,.7);
\end{tikzpicture}\caption{Pseudoline arrangement for $w$.}
\end{figure}

 By our assumption, $y$ lies in $SO_{2n+1}(\C)$ and we know that matrix
 coefficients of $y$ of a certain type \eqref{e:allowable} are positive. Indeed, a chamber minor 
 $\tilde\Delta_{J}(y)=\left<e_{j_1}\wedge\cdots\wedge e_{j_k},\ y\cdot v_{\omega_k}^+\right>$ is of this 
 allowable type precisely if 
 $v=e_{j_1}\wedge\cdots\wedge e_{j_k}$ lies in a $1$-dimensional weight space 
 of  the restricted representation, $Res^{SL_{2n+1}}_{SO_{2n+1}} V_{\tilde\omega_k}$.


All the weight spaces of fundamental representations $V_{\tilde\omega_k}$
of  $SL_{2n+1}(\C)$ are $1$-dimensional. Furthermore, the weights which 
stay non-zero when we restrict to $SO_{2n+1}(\C)$ all stay distinct. Therefore
their weight spaces stay $1$-dimensional. (Whereas the zero weight space 
of $Res^{SL_{2n+1}}_{SO_{2n+1}} V_{\tilde\omega_k}$ becomes potentially 
higher dimensional). 
Now the weight vector $e_{j_1}\wedge\dotsc\wedge e_{j_k}$ in $V_{\tilde\omega_k}$ 
has weight $\tilde\ep_{j_1}+\dotsc+\tilde \ep_{j_k}$, which restricts to
a non-zero weight of the torus $T$ of $SO_{2n+1}$ precisely if the set $J$
of indices  is `asymmetric' about $n+1$, so if there is some $m\in J$ for which $2n+2-m$
is not in $J$. 

The following Claim implies that  the chamber minors of our reduced expression 
 $w_0=w^n$ all have this property. Therefore $\Delta_J(y)>0$ for these minors,
by \eqref{e:allowable}. And therefore $y$ is totally positive, as desired.  

\vskip .2cm

\noindent{\it Claim:} Every chamber in the pseudo-line arrangement associated
to the reduced expression $w^n$ of $w_0$ lies between lines labeled $k$
and $2n+2-k$ for some $k$.

\vskip .2cm
\noindent{\it Proof of the Claim:} 
The pseudo-line arrangement is made up of $n$ copies of the one in Figure~\ref{f:firstw}.
The $j$-th copy is illustrated in Figure~\ref{f:jthw}.
Any chamber in this part of the 
pseudo-line arrangement either lies in between the
lines labeled $j$ and $2n+2-j$, or between
the lines $j+1$ and $2n+1-j$.
This proves the claim.

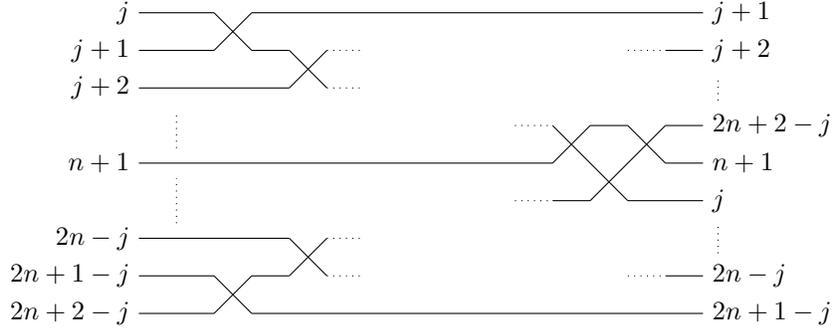
\begin{figure}\label{f:jthw}
\begin{tikzpicture}
\draw (0,3.5)-- (1,3.5);
\node [left] at (0,3.5) {$j$};
\node [left] at (0, 3) {$j+1$}; 
\draw (0,3)-- (1,3); 
\node [left] at (0, 2.5) {$j+2$};
\draw (0,2.5)-- (2,2.5);
\draw (2,2.5)--(2.5,3);
\draw (2.5,2.5)--(2,3);
\draw (1.5,3)--(2,3);
\draw [dotted] (.5, 1.7) -- (.5,2.2);
\node [left] at (0, 1.5) {$n+1$\ };
\draw (0,1.5)-- (5.5,1.5);
\draw [dotted] (.5, .7) -- (.5,1.3);
\draw (0,.5)--(2,.5);
\draw (0,0)--(1,0);
\node [left] at (0,0) {$2n+1-j$}; 
\node [left] at (0,.5) {$2n-j$};
\draw (1,3.5) -- (1.5,3);
\draw (1,3) -- (1.5,3.5);
\draw (1,-.5) -- (1.5, 0);
\draw (1,0)-- (1.5, -.5);
\draw (1.5,3.5)--(7.5,3.5);
\draw (1.5,-.5) -- (7.5,-.5);
\draw (0,-.5) -- (1,-.5);
\node [left] at (0,-.5) {$2n+2-j$}; 
\draw (1.5,0)--(2,0);
\draw (2,0)--(2.5,.5);
\draw (2.5,0)-- (2,.5);
\draw [dotted] (2.5,0)--(3,0);
\draw [dotted]  (2.5,.5)--(3,.5);
\draw [dotted] (2.5,2.5) -- (3,2.5);
\draw [dotted] (2.5,3) -- (3,3);
\draw (5.5,1.5)--(6,2);
\draw (5.5,2)--(6.5,1);
\draw (6.5,1)--(7.5,1);
\draw (6,1)--(7,2);
\draw (6,2)--(6.5,2);
\draw (6.5,2)--(7,1.5);
\draw (7,1.5)--(7.5,1.5);
\draw (5.5,1)--(6,1);
\draw (7,2)--(7.5,2);
\draw [dotted] (6.5,0) -- (7,0);
\draw (7,0) -- (7.5,0);
\draw [dotted] (6.5,3) -- (7,3);
\draw (7,3) -- (7.5,3);
\draw [dotted] (5,2)--(5.5,2);
\draw [dotted] (5,1)--(5.5,1);
\node [right] at (7.5,-.5) {$2n+1-j$};
\node [right] at (7.5,0) {$2n-j$};
\node [right] at (7.5,1) {$j$};
\node [right] at (7.5,1.5) {$n+1$};
\node [right] at (7.5,2) {$2n+2-j$};
\node [right] at (7.5,3) {$j+2$};
\node [right] at (7.5,3.5) {$j+1$};
\draw [dotted] (7.7,2.6)--(7.7,2.3);
\draw [dotted] (7.7,.3)--(7.7,.7);
\end{tikzpicture}
\caption{The $j$-th segment of the pseudo-line arrangement for $w_0=w^n$ in the proof for type $B_n$.}
\end{figure}

\subsection{Type $C_n$}
The order of the weight lattice modulo the root lattice (index of connection) is 2.  Let us choose a basis $\ep_1,\ep_2,\ldots,\ep_n \in \mathfrak t_\R^*$ so that the long simple root is $\alpha_1 = 2\ep_1$, and the short simple roots are $\alpha_k = \ep_k - \ep_{k-1}$ for $2 \leq k \leq n$.
Note that $\omega_{n-1} = \ep_{n-1}+\ep_n$ is allowable, while $\omega_n=\ep_n$ is not.  
For the fundamental representations $V_{\omega_i}$ for $1\le i\le n-2$, which  are not representations of the adjoint group,
we consider $V_{\omega_i + \omega_n}  \subset V_{\omega_i} \otimes V_{\omega_n}$. 
(Note that since $\omega_n$ is also not allowable and the index of connection is $2$, $\omega_i+\omega_n$ is allowable.)
Then for any $w\in W$ we have
$$
\ip{\dot w\cdot v^+_{\omega_i}, y\cdot v^+_{\omega_i}}\ip{ \dot w \cdot v^+_{\omega_n},  v^+_{\omega_n}}=\ip{\dot w\cdot (v^+_{\omega_i}\otimes v^+_{\omega_n}),y\cdot (v^+_{\omega_i}\otimes v^+_{\omega_n})}>0,
$$
by assumption \eqref{e:allowable} on $y$. It remains to show that $\ip{\dot w \cdot v_{\omega_n},y \cdot v_{\omega_n}}>0$,
since then $\ip{\dot w \cdot v_{\omega_i},y \cdot v_{\omega_i}}>0$ for all $w\in W$ and fundamental weights $\omega_i$,
whereby $y$ has to be totally positive, because of Proposition~\ref{P:BZ}.


We now consider $V = V_{\omega_n}$, which is $2n$-dimensional with weights $\pm \ep_k$ for $1 \leq k \leq n$. 
We have the following:
\begin{lem} \label{L:Clemma} \
\begin{enumerate}
\item
The equivalence relation on the weights of $V=V_{\omega_n}$ generated by $\la \sim \mu$ if $\la + \mu \in W \cdot \omega_{n-1}$ has a single equivalence class.
\item
$\omega_{n-1}$ appears as a weight in $V_{2\omega_n}$ with multiplicity 1.  The weight
$\omega_{n-1}$ appears as a weight in $V \otimes V$ with multiplicity 2.
\item
$V_{\omega_{n-1}}$ occurs as an irreducible factor of $V \otimes V$ with multiplicity 1.
\end{enumerate}
\end{lem}
\begin{proof}
We have $W \cdot \omega_{n-1} = \{\pm \ep_i \pm \ep_j \mid 1 \leq i < j \leq n\}$.  (1) follows by inspection.   The first statement of (2) follows from the fact that $2\omega_n - \omega_{n-1} = \alpha_n$ is a simple root.  The second statement of (2) follows by inspection of the weights of $V$.  (3) follows from the fact that there are no weights $\mu$ satisfying $2\omega_n > \mu > \omega_{n-1}$ in dominance order.
\end{proof}

We may now proceed as in the proof for $A_n$ for $n$ odd.  We consider the inclusion $U = V_{2\omega_n} \subset V \otimes V = Z$ and look at a vector $z \in U$ with weight $\nu = \la + \mu \in W \cdot \omega_{n-1}$.  We first argue that $z$ can be chosen so that $\ip{z,x\cdot v_{2\omega_n}^+} >0$ for all totally positive $x \in U^-_{>0}$.  In \cite[Corollary 7.2]{BeZe:Chamber}, Berenstein and Zelevinsky show that there is an inclusion $Sp_{2n}(\C) \to SL_{2n}(\C)$ such that the image of the totally positive part $U^-_{>0}$ of the unipotent of $Sp_{2n}$ lies in the totally nonnegative part of $SL_{2n}$.  Now, $V$ is the standard representation of $SL_{2n}$ and contains the irreducible representation $\Sym^2(V)$.  The restriction of $\Sym^2(V)$ to $Sp_{2n}(\C)$ contains the representation $U$, and $v_{2\omega_n}^+$ is exactly the highest-weight vector of $\Sym^2(V)$.  By Remark \ref{rem:can}, we can choose weight vectors $z \in \Sym^2(V)$ such that $\ip{z,x\cdot v_{2\omega_n}^+} > 0$ for all $x$ which are totally positive in the unipotent of $SL_{2n}$.  It follows that $\ip{z,x\cdot v_{2\omega_n}^+} > 0$ for $x \in U^-_{>0}$.

Under the inclusion $U \subset Z$, the vector $z$ is a linear combination of $v_\la \otimes v_\mu$ and $v_\mu \otimes v_\la$ by Lemma \ref{L:Clemma}(2).  Here $\la, \mu \in W \cdot \omega_n$, and if $\la = w \cdot \omega_n$, then $v_\la = \dot w \cdot v_{\omega_n} \in V$ and similarly for $\mu$. 
We have $z=A v_\la \otimes v_\mu + B v_\mu \otimes v_\la$ for positive $A,B$.  Using Lemma \ref{L:Clemma}(3), we obtain that
\begin{align*}
0 &< \ip{z, y \cdot v^+_{2\omega_n}}_U\\
&= \ip{A v_\la \otimes v_\mu + B v_\mu \otimes v_\la, y \cdot (v^+_{\omega_n} \otimes v^+_{\omega_n})}_Z\\
&=
(A+B) \ip{v_{\la},y \cdot v^+_{\omega_n}}_V \,\ip{v_\mu,y \cdot v^+_{\omega_n}}_V
\end{align*}
It follows that $\ip{v_{\la},y \cdot v^+_{\omega_n}}$ and $\ip{v_{\mu},y \cdot v^+_{\omega_n}}$ have the same sign.  By Lemma \ref{L:Clemma}(1)
$\lambda, \mu$ can be any two weights in $W\cdot\omega_n$ in the arguments above,  therefore $\ip{v_{\la},y \cdot v^+_{\omega_n}}$ has the same sign
as $\ip{v^+_{\omega_n},y \cdot v^+_{\omega_n}}=1$. This concludes the proof in the $C_n$ case.

\subsection{Type $D_n$}
We take as simple roots $\alpha_1 = \ep_1 + \ep_2$ and $\alpha_k = \ep_k-\ep_{k-1}$ for $2 \leq k \leq n$.  Let us consider the (spin) representation $V = V_{\omega_1}$ with highest weight $1/2(\ep_1+\ep_2 + \cdots \ep_n)$.  The argument is the same as for $C_n$ (using also Remark \ref{rem:can}), after the following Lemma.  Note that $\omega_3$ is allowable.

\begin{lem} \
\begin{enumerate}
\item
The equivalence relation on the weights of $V$ generated by $\la \sim \mu$ if $\la + \mu \in W \cdot \omega_{3}$ has a single equivalence class.
\item
$\omega_{3}$ appears as a weight in $V_{2\omega_1}$ with multiplicity 1.  The weight
$\omega_{3}$ appears as a weight in $V \otimes V$ with multiplicity 2.
\item
$V_{\omega_{3}}$ occurs as an irreducible factor of $V \otimes V$ with multiplicity 1.
\end{enumerate}
\end{lem}
\begin{proof}
The representation $V$ has dimension $2^{n-1}$, with weights the even signed permutations of the vector $(1/2,1/2,\ldots,1/2) \in \R^n$.  The rest of the argument is identical to the proof of Lemma \ref{L:Clemma}.
\end{proof}

\subsection{Type $E_6$}
The index of connection of $E_6$ is 3, which is odd.  The proof for $A_n$ with $n$ even can be applied here essentially verbatim.

\subsection{Type $E_7$}
We fix a labelling of the Dynkin diagram by letting $7$ label the minuscule node (at the end of the long leg), and $6$ be the unique node adjacent to $7$.  We note that $\omega_6$ is allowable.  The argument is the same as for $C_n$ (using also Remark \ref{rem:can}), after the following Lemma.

\begin{lem} \
\begin{enumerate}
\item
The equivalence relation on the weights of $V$ generated by $\la \sim \mu$ if $\la + \mu \in W \cdot \omega_{6}$ has a single equivalence class.
\item
$\omega_{6}$ appears as a weight in $V_{2\omega_7}$ with multiplicity 1.  The weight
$\omega_{6}$ appears as a weight in $V \otimes V$ with multiplicity 2.
\item
$V_{\omega_{6}}$ occurs as an irreducible factor of $V \otimes V$ with multiplicity 1.
\end{enumerate}
\end{lem}
\begin{proof}
Can be verified by computer, which we did using John Stembridge's {\tt coxeter/weyl} package.
\end{proof}

\subsection{Types $E_8$, $F_4$, and $G_2$}
The adjoint group is simply-connected, so there is nothing to prove here.

\appendix

\section{Quantum Schubert positivity implies affine Schubert positivity in type $C$}\label{s:proof1}

%
We shall need the quantum Chevalley formula of $qH^*(G/B)$, due to Peterson \cite{Pet:QCoh} and Fulton-Woodward \cite{FuWo:SchubProds}.

For $w \in W$, define $\pi_P(w):=w_1$, where $w=w_1w_2$ with $w_1\in W^P$
and $w_2\in W_P$.  Also we have that $2\rho$ is the sum of positive roots and set $2\rho_P :=
\sum_{\alpha \in \Delta_{P,+}} \alpha$. Let $Q_P^\vee$ be the sublattice of $Q^\vee$
spanned by the simple coroots $\alpha_j^\vee$ for $j\in I_P$, and let $\eta_P:Q^\vee\to
Q^\vee/Q_P^\vee$ be the natural projection.  We let $w \gtrdot v$ denote a cover in Bruhat order.

\begin{thm}[Quantum equivariant Chevalley formula \cite{Pet:QCoh,FuWo:SchubProds}]
\label{t:qChev} Let $i \in I^P$ and $w \in W^P$.  Then
we have in $qH^*(G/P)$
\begin{equation}\label{e:qChev}
\sigma^P_{s_i} \, \sigma^P_{w} =  \sum_\alpha \ip{\alpha^\vee,
\omega_i}\,\sigma^P_{w r_\alpha} + \sum_\alpha
\ip{\alpha^\vee,\omega_i}\,q_{\eta_P(\alpha^\vee)}\,\sigma^P_{\pi_P(wr_\alpha)}
\end{equation}
where the first summation is over $\alpha \in \Delta^P_+$
such that $wr_\alpha \gtrdot w$ and $wr_\al\in W^P$, and the second
summation is over $\alpha \in \Delta^P_+$ such that
$\ell(\pi_P(wr_\alpha)) = \ell(w) + 1 - \ip{\alpha^\vee, 2(\rho-\rho_P)}$.
\end{thm}
It is known that in the second summation, we only need to sum over $\alpha$ such that $\ell(r_\alpha) = \ip{\alpha^\vee,2\rho}-1$.

%

We now let $G$ be of type $C_n$.  We choose conventions so that the $\alpha_1,\ldots,\alpha_{n-1}$ are short and $\alpha_n$ is long.  One may check that the positive coroots $\alpha^\vee$ satisfying $\ell(r_\alpha) = \ip{\alpha^\vee,2\rho}-1$ are exactly those of the form $\alpha^\vee_i + \alpha^\vee_{i+1} + \cdots + \alpha^\vee_{j}$.

\begin{prop}
Conjecture \ref{c:QP} holds in type $C_n$.
\end{prop}
\begin{proof}
Since we already know that we have $X^\af_{>0} = X_{>0}$ and $X^\af_{>0} \subseteq X^\q_{>0}$, it suffices to show that any quantum Schubert positive point is also affine Schubert positive.  By Theorem \ref{t:PetIsoLoop}, it suffices to show that if $x \in X^\q_{>0}$ then $q_i(x) > 0$ for each $i \in I$.

Now let $i \in I$, and let $v_i$ be the longest element in $W^{P_i}$, where $P_i$ is the maximal parabolic labeled by $i$.  Let us consider the product $\sigma^{s_i} \,\sigma^{v_i}$ and apply Theorem \ref{t:qChev} for the base $P = B$.  Since $v_i\alpha < 0$ for any $\alpha \in \Delta_+^{P_i}$, we see that the first summation of \eqref{e:qChev} is empty.  We note that $\ell(v_ir_\alpha) = \ell(v_i) -\ell(r_\alpha)$ if and only if $r_\alpha \in W^{P_i}$.  The only such coroots $\alpha^\vee$ which also satisfy $\ell(r_\alpha) = \ip{\alpha^\vee,2\rho}-1$ are $\alpha^\vee_i$ and $\beta^\vee_i:=\alpha^\vee_i + \cdots + \alpha^\vee_n$ in the case $1 \leq i \leq n-1$, and if $i = n$ we only have $\alpha^\vee_n$.  Thus we obtain 
\begin{align*}
\sigma^{s_i}\,\sigma^{v_i} &= q_i \sigma^{v_is_i} +  q_iq_{i+1}\cdots q_n \sigma^{v_ir_{\beta_i^\vee}}& \mbox{for $i \neq n$}\\
\sigma^{s_n}\,\sigma^{v_n} &= q_n \sigma^{v_is_i}
\end{align*}
It follows that if $\sigma^{w}(x) > 0$ for all $w \in W$ then $q_i(x)>0$ for all $i$.
\end{proof}





\begin{thebibliography}{10}

%
%
\bibitem{BeFoZe:TotPos}
A. Berenstein, S. Fomin, and A. Zelevinsky, \emph{Parametrizations
  of canonical bases and totally positive matrices}, Adv. Math. \textbf{122}
  (1996), no.~1, 49--149.

\bibitem{BeZe:Chamber}
A.~Berenstin and A.~Zelevinsky, \emph{Total positivity in Schubert varieties}, Comment. Math. Helv.  72  (1997),  no. 1, 128--166.


\bibitem{BFM:K-hom}
R.~Bezrukavnikov, M.~Finkelberg, and I.~Mirkovic, \emph{Equivariant
$K$-homology of affine Grassmannian and Toda lattice}, Compos. Math. 141 (2005), no. 3, 746--768.

\bibitem{Borel:CohG/P}
A.~Borel, \emph{Sur la cohomologie des espaces fibr\'es principaux et des
  espaces homog\`enes de groupes de {L}ie compacts}, Ann. of Math. (2)
57 (1953), 115--207.
%

\bibitem{Cio:QCohFl}
I.~Ciocan-Fontanine, \emph{Quantum cohomology of flag varieties}, Internat.
  Math. Res. Notices (1995), no.~6, 263--277.
%
%

\bibitem{Edr:ToeplMat}
A.~Edrei, \emph{Proof of a conjecture of {S}choenberg on the generating
  function of a totally positive sequence}, Canadian J. Math. 5
  (1953), 86--94.

%

\bibitem{Fom} S.~Fomin, \emph{Total positivity and cluster algebras},
Proceedings of the International Congress of Mathematicians, vol. 2, 2010, 125--145. 


\bibitem{FuWo:SchubProds}
W.~Fulton and C.~Woodward,
\emph{On the quantum product of Schubert classes}, J. Algebraic Geom. 13 (2004), no. 4, 641--661. 


\bibitem{Gin:GS}
V.~Ginzburg, \emph{Perverse sheaves on a loop group and Langland's
duality}, {\tt arXiv:alg-geom/9511007}.

\bibitem{GiKi:FlTod}
A.~Givental and B.~Kim, \emph{Quantum cohomology of flag manifolds and Toda
  lattices}, Comm. Math. Phys. 168 (1995), 609--641.
%

\bibitem{Hartshorne} R.~Hartshorne, \emph{Algebraic Geometry}, Graduate Texts in Mathematics, No. 52. Springer-Verlag, New York-Heidelberg, 1977. xvi+496 pp.

\bibitem{Kim} B.~Kim, \emph{Quantum cohomology of flag manifolds $G/B$ and quantum Toda  lattices}, Ann. of Math. (2)  149  (1999),  no. 1, 129--148.

%
%
%



\bibitem{Kos:qToda}
\bysame, \emph{{Q}uantization and representation theory}, Proc. Oxford
  Conference on Group Theory and Physics, Oxford, 1977, pp.~287--316.


\bibitem{Kos:Toda}
B.~Kostant, \emph{The solution to a generalized Toda lattice and representation theory},  Adv. in Math.  34  (1979), no. 3, 195--338.

\bibitem{Kos:QCoh}
\bysame, \emph{Flag manifold quantum cohomology, the Toda lattice, and the
  representation with highest weight $\rho$}, Selecta Math. (N.S.) 2
  (1996), 43--91.
%
\bibitem{Kos:QCoh2}
\bysame, \emph{Quantum cohomology of the flag manifold as an algebra of
  rational functions on a unipotent algebraic group}, Deformation theory and
  symplectic geometry (Ascona, 1996), Math. Phys. Stud., vol.~20, Kluwer Acad.
  Publ., Dordrecht, 1997, pp.~157--175.

\bibitem{Kum:Book}
S.~Kumar, \emph{Kac-Moody Groups, their Flag Varieties and
Representation Theory}, Progress in Mathematics, 204. Birkhäuser Boston, Inc., Boston, MA, 2002. xvi+606 pp.

\bibitem{KuNo:pos}
S.~Kumar and M.V.~Nori,
\emph{Positivity of the cup product in cohomology of flag varieties associated to Kac-Moody groups},
Internat. Math. Res. Notices 1998, no. 14, 757--763. 


\bibitem{LaSh:QH}
T.~Lam and M.~Shimozono, \emph{Quantum cohomology of $G/P$ and
homology of affine Grassmannian}, Acta. Math. 204 (2010), 49--90.
%

\bibitem{Lus:GS} G.~Lusztig, \emph{Singularities, character formulas, and a
q-analogue for weight multiplicities}, in Analyse et topologie sur
les espaces singuliers, Asterisque 101-102 (1982),
208--229.

\bibitem{Lus:TotPos94}
\bysame, \emph{Total positivity in reductive groups}, Lie theory and
  geometry: in honor of Bertram Kostant (G.~I. Lehrer, ed.), Progress in
  Mathematics, vol. 123, Birkhaeuser, Boston, 1994, pp.~531--568.

\bibitem{Lus:Can}
\bysame, \emph{Total positivity and canonical bases}, Algebraic groups and Lie
  groups, Cambridge Univ. Press, Cambridge, 1997, pp.~281--295.


%




\bibitem{Minc:NonnegMat}
H. Minc, \emph{Nonnegative matrices}, John Wiley \& Sons Inc., New York,
  1988.

\bibitem{MV:GS}
I.~Mirkovic and K.~Vilonen, \emph{Geometric Langlands duality and representations of algebraic groups over commutative rings}, Ann. of Math. (2) 166  (2007),  no. 1, 95--143.


\bibitem{Pet:QCoh}
D.~Peterson, \emph{Quantum cohomology of ${G}/{P}$}, Lecture Course, M.I.T.,
  Spring Term, 1997.

%
\bibitem{PS:loopgroups}
A.~Pressley and G.~Segal, \emph{Loop groups}, Oxford Mathematical Monographs. Oxford Science Publications. The Clarendon Press, Oxford University Press, New York, 1986. viii+318 pp.
%
\bibitem{Qui:unp}
D. Quillen, unpublished.

\bibitem{Rie:QCohGr}
K.~Rietsch, \emph{Quantum cohomology of {G}rassmannians and total positivity},
  Duke Math. J. \textbf{113} (2001), no.~3, 521--551.

\bibitem{Rie:QCohPFl}
K.~Rietsch, \emph{Totally positive Toeplitz matrices and quantum cohomology of partial flag varieties},  J. Amer. Math. Soc. 16  (2003), no. 2, 363--392.



\bibitem{Rie:TotPosGBCKS}
\bysame, \emph{A mirror construction for the totally nonnegative part of the
  {P}eterson variety}, Nagoya Mathematical Journal (2006), 105--142, 



\bibitem{Rie:MSgen}
\bysame, \emph{A mirror symmetric construction for
  $q{H}^*_{T}({G}/{P})_{(q)}$}, Advances in Mathematics \textbf{217} (2008),
  2401--2442. 	arXiv:math.AG/0511124v2

\bibitem{Rie:QToda}
\bysame, \emph{A mirror symmetric solution to the quantum {T}oda lattice},
Communications in Mathematical Physics, Volume 309, Issue 1, pp.23-49,  
 DOI: 10.1007/s00220-011-1308-8  

\bibitem{Shafarevich:AG1} 
I.R. Shafarevich, \emph{Basic algebraic geometry. 1. Varieties in projective space. Second edition.} Springer-Verlag, Berlin, 1994. xx+303 pp. 


%

\end{thebibliography}

\def\cprime{$'$}
\providecommand{\bysame}{\leavevmode\hbox to3em{\hrulefill}\thinspace}

\end{document}